\documentclass[11pt,reqno]{amsart}
\usepackage{amsmath, amssymb, amsthm}
\usepackage{tensor}
\usepackage{enumerate}
\usepackage[colorlinks=true,allcolors=blue]{hyperref}
\usepackage{blkarray}
\usepackage[numbers]{natbib}
\usepackage[font=footnotesize]{caption}

\usepackage{tikz}
\usetikzlibrary{cd,arrows,calc}

\newcommand{\C}{\mathbb{C}}

\newcommand{\Z}{\mathbb{Z}}
\newcommand{\N}{\mathbb{N}}

\newcommand{\bK}{\mathbb{K}}
\newcommand{\bT}{\mathbb{T}}

\newcommand{\Cuntz}[1]{\mathcal{O}_{#1}}
\newcommand{\JiangSu}{\mathcal{Z}}

\newcommand{\Aut}[1]{\text{\normalfont Aut}(#1)}

\newcommand{\eqAut}[2]{\operatorname{Aut}_{#1}(#2)}
\newcommand{\eqAutId}[2]{\operatorname{Aut}_{#1,0}(#2)}
\newcommand{\eqIsom}[3]{\operatorname{Iso}_{#1}\left(#2, #3\right)}
\newcommand{\Endo}[1]{\operatorname{End}\left(#1\right)}

\newcommand{\Proj}[2]{\operatorname{Proj}_{#1}(#2)}
\newcommand{\Rep}[1]{\operatorname{Rep}(#1)}

\newcommand{\id}[1]{\operatorname{id}_{#1}}
\newcommand{\obj}[1]{\text{obj}(#1)}

\newcommand{\cI}{\mathcal{I}}

\newcommand{\lscal}[3]{\tensor*[_{#1}]{\langle #2, #3 \rangle}{}}

\newcommand{\bddR}{R_{\text{\normalfont bdd}}}

\newcommand{\VecCf}{\text{Vec}_{\C}^\text{fin}}
\DeclareMathOperator{\Ad}{Ad}

\DeclareMathOperator{\hocolim}{hocolim}

\newtheorem{theorem}{Theorem}[section]
\newtheorem{lemma}[theorem]{Lemma}
\newtheorem{corollary}[theorem]{Corollary}
\newtheorem{prop}[theorem]{Proposition}

\theoremstyle{definition}
\newtheorem{definition}[theorem]{Definition}
\newtheorem{remark}[theorem]{Remark}

\makeatletter
\newcommand{\extp}{\@ifnextchar^\@extp{\@extp^{\,}}}
\def\@extp^#1{\mathop{\bigwedge\nolimits^{\!#1}}}
\makeatother

\begin{document}
\title[Higher equivariant Dixmier-Douady Theory]{Equivariant higher Dixmier-Douady Theory \\ for circle actions on UHF-algebras}
\author{David E.\ Evans \and Ulrich Pennig}
\address{Cardiff University, School of Mathematics, Senghennydd Road, Cardiff, CF24 4AG, Wales, UK}
\email{EvansDE@cardiff.ac.uk}
\email{PennigU@cardiff.ac.uk}

\begin{abstract}
We develop an equivariant Dixmier-Douady theory for locally trivial bundles of $C^*$-algebras with fibre $D \otimes \bK$ equipped with a fibrewise $\bT$-action, where $\bT$ denotes the circle group and $D = \Endo{V}^{\otimes \infty}$ for a $\bT$-representation $V$. In particular, we show that the group of $\bT$-equivariant $*$-automorphisms $\eqAut{\bT}{D \otimes \bK}$ is an infinite loop space giving rise to a cohomology theory $E^*_{D,\bT}(X)$. Isomorphism classes of equivariant bundles then form a group with respect to the fibrewise tensor product that is isomorphic to $E^1_{D,\bT}(X) \cong [X, B\!\eqAut{\bT}{D \otimes \bK}]$. We compute this group for tori and compare the case $D = \C$ to the equivariant Brauer group for trivial actions on the base space. 
\end{abstract}

\maketitle

\section{Introduction} \label{sec:Intro}
Continuous fields of $C^*$-algebras have found applications in various different areas: They arise naturally in representation theory \cite{paper:HigsonMackeyMachine}, index theory~\cite{paper:HigsonTanGrpoid}, twisted $K$-theory \cite{paper:AtiyahSegal,paper:KaroubiTwistedK,paper:FreedHopkinsTeleman} and conformal field theory \cite{paper:EvansGannon-ModInv,paper:EvansGannon-ModInvII}. In fact, by the Dauns-Hofmann theorem (see for example \cite[Thm.~A.34]{book:RaeburnWilliams}) any $C^*$-algebra $B$ with Hausdorff primitive spectrum $X$ is a continuous field of simple $C^*$-algebras over $X$.  
While the classification of all continuous fields of simple $C^*$-algebras over a given topological space $X$ is out of reach, section algebras of locally trivial bundles provide a particularly well-behaved class of fields that is open to classification by methods from homotopy theory: In fact, any such bundle with fibre algebra $A$ is associated to a principal $\Aut{A}$-bundle $P \to X$, where $\Aut{A}$ denotes the group of $*$-automorphisms of $A$ equipped with the pointwise-norm topology. But for each such principal bundle there exists a continuous map $f \colon X \to B\Aut{A}$, unique up to homotopy, such that 
\[
	\begin{tikzcd}
		P \ar[d] \ar[r] & E\Aut{A} \ar[d] \\
		X \ar[r, "f" below] & B\Aut{A}
	\end{tikzcd}
\]
is a pullback diagram, where $E\Aut{A} \to B\Aut{A}$ is the universal principal $\Aut{A}$-bundle over the classifying space $B\Aut{A}$. This reduces the classification of locally trivial $C^*$-algebra bundles with fibre $A$ to the computation of the homotopy set $[X, B\Aut{A}]$. 

The classifying space $BG$ is a delooping of the topological group $G$ in the sense that the based loop space $\Omega BG$ is homotopy equivalent to $G$. 
This implies $\pi_n(BG) \cong \pi_{n-1}(G)$. While this observation allows us to classify the isomorphism classes of principal $\Aut{A}$-bundles over spheres and other suspensions of spaces, in general the homotopy type of $B\Aut{A}$ is difficult to determine (e.g., in the commutative case $\Aut{C(X)} \cong \text{Homeo}(X)$). This changes drastically for fibre algebras $A = D \otimes \bK$, where $D$ belongs to the class of strongly self-absorbing $C^*$-algebras \cite{paper:TomsWinter} and $\bK$ denotes the compact operators on a separable infinite dimensional Hilbert space. Such algebras satisfy $A \otimes A \cong A$ with controllable isomorphisms that equip both  homotopy sets $[X, \Aut{A}]$ and $[X, B\Aut{A}]$ with an extra structure induced by the tensor product. 

It was shown by Dadarlat and the second author that the tensor product gives $\Aut{D \otimes \bK}$ the structure of an infinite loop space \cite{paper:DadarlatP-DD-theory,paper:DadarlatP-UnitSpectra}. This means that it has deloopings of arbitrary order, or in other words, that there is a sequence of spaces $(Y_n)_{n \in \N_0}$ such that $Y_0 \simeq \Aut{D \otimes \bK}$ and $\Omega Y_{k+1} \simeq Y_k$. The sequence $(Y_n)_{n \in \N_0}$ together with the equivalences $Y_k \to \Omega Y_{k+1}$ is called an $\Omega$-spectrum, and it gives rise to a cohomology theory $E^k(X) = [X,Y_k]$. 

The space $\Aut{D \otimes \bK}$ now has two a priori different deloopings: The first one given by the space $Y_1$ in the $\Omega$-spectrum underlying the infinite loop space structure induced by the tensor product, which should better be denoted by $B_{\otimes} \Aut{D \otimes \bK}$, and the second one by considering $\Aut{D \otimes \bK}$ as a topological group and forming its classifying space $B\Aut{D \otimes \bK}$. It was shown in \cite{paper:DadarlatP-UnitSpectra} 
that $B_{\otimes}\Aut{D\otimes \bK} \simeq B\Aut{D\otimes \bK}$. Thus, the cohomology theory $E_D^*(X)$ represented by the $\Omega$-spectrum satisfies 
\[
	E_D^0(X) = [X,\Aut{D \otimes \bK}] \qquad \text{and} \qquad E^1_D(X) \cong [X,B\Aut{D \otimes \bK}]\ .
\]
As a result $E^1_D(X)$ classifies locally trivial $D \otimes \bK$-bundles up to isomorphism, and - as the first group in a cohomology theory - is amenable to the computational power of algebraic topology. In particular, it is computable via the Atiyah-Hirzebruch spectral sequence. 

Apart from the classification result itself, the groups $E^1_{\Cuntz{\infty}}(X)$ also provide a natural home for invariants of locally trivial bundles with fibre the Cuntz algebras $\Cuntz{n}$ or Cuntz-Toeplitz algebras  \cite{paper:SogabeToeplitz,paper:SogabeCuntz}. If $\alpha \colon G \to \Aut{B}$ is an action of a Poly-$\Z$ group on a Kirchberg algebra $B$, then the associated bundle $EG \times_{\alpha} \Aut{B} \to BG$ played a crucial role in the classification of such actions developed in \cite{paper:IzumiMatuiPolyZ_I, paper:IzumiMatuiPolyZ_II}.

The present paper was motivated by a construction in equivariant twisted $K$-theory \cite{preprint:EvansP}: Let $\VecCf$ be the groupoid of finite-dimensional complex inner product spaces and unitary isomorphisms. An exponential functor $F \colon (\VecCf,\oplus) \to (\VecCf,\otimes)$ in the sense of \cite[Def.~2.2]{preprint:EvansP} associates to a finite-dimensional representation $W$ of a group $G$ another representation $V = F(W)$ and an infinite tensor product $C^*$-algebra $D$ defined by 
\[
	D = \Endo{V}^{\otimes \infty}\ .
\] 
Infinite tensor products like this are called uniformly hyperfinite (UHF). The group $G$ acts on $D$ by conjugation in each tensor factor. The algebra~$D$ tensorially absorbs $\Endo{F(W')}$ for any subrepresentation $W'$ of $W$, ie.\ the exponential functor induces a $*$-isomorphism
\(
	\Endo{F(W')} \otimes D \cong D
\)
that turns $F(W') \otimes D$ into a Morita equivalence. This is used in \cite{preprint:EvansP} in the case $G = SU(n)$ to construct a Fell bundle $\mathcal{E}$ whose associated $C^*$-algebra $C^*(\mathcal{E})$ is a continuous field over $G$ with fibre $D$. The bundle $\mathcal{E}$ boils down to the basic gerbe used in \cite{paper:FreedHopkinsTeleman} in case $F$ is the determinant functor. The algebra $C^*(\mathcal{E})$ carries a $G$-action that is compatible with the adjoint action of $G$ on itself. Moreover, $C^*(\mathcal{E}) \otimes \bK$ is isomorphic to the section algebra of a locally trivial bundle $\mathcal{A} \to G$ with fibre $D \otimes \bK$ and therefore falls into the scope of the classification results mentioned earlier giving a class $[\mathcal{A}] \in E^1_D(G)$. However, this topological classification does not take the $G$-action on $\mathcal{A}$ (inherited from the one on $C^*(\mathcal{E}) \otimes \bK$) into account. 

In the present paper we will initiate a programme with the goal to develop an equivariant extension of  the generalised Dixmier-Douady theory in \cite{paper:DadarlatP-DD-theory}. As a starting point we will simplify the situation by only considering equivariant locally trivial bundles $\mathcal{A} \to X$ with fibre $D \otimes \bK$ for a UHF-algebra~$D$ as above and a fibrewise action of the circle group $\bT$. While the action we consider on $D$ arises from an infinite tensor product of representations, we take $\bK = \bK(H)$ with an infinite-dimensional separable Hilbert space $H$ that contains all $\bT$-representations with infinite multiplicity. The structure group of such bundles reduces to the group $\eqAut{\bT}{D \otimes \bK}$ of $\bT$-equivariant $*$-automorphisms. Examples arise from the  construction in \cite{preprint:EvansP} for $G = SU(2)$ after pulling the bundle back to the maximal torus of $G$. 

In our setting we are able to describe the complete picture. Our main result, proven as Cor.~\ref{cor:coh_theory} can be summarised as follows:
\begin{theorem}
	The topological group $\eqAut{\bT}{D \otimes \bK}$ is an infinite loop space with associated cohomology theory $E_{D,\bT}^*(X)$ that satisfies 
	\[
		E^0_{D,\bT}(X) = [X, \eqAut{\bT}{D\otimes \bK}] \quad \text{and} \quad E^1_{D,\bT}(X) \cong [X, B\!\eqAut{\bT}{D\otimes \bK}]\ .
	\]
	In particular, isomorphism classes of $\bT$-equivariant locally trivial $C^*$-algebra bundles with fibres isomorphic to the $\bT$-algebra $D \otimes \bK$ over a finite CW-complex $X$ with trivial $\bT$-action on $X$ form a group with respect to the tensor product that is isomorphic to $E^1_{D,\bT}(X)$.
\end{theorem} 

Along the way we also compute the coefficients of the cohomology theory $E_{D,\bT}^*(X)$, which boil down to the homotopy groups $\pi_n(\eqAut{\bT}{D \otimes \bK})$. In the non-equivariant case the groups $\pi_n(\Aut{A})$ for UHF-algebras $A$ have been studied first by Thomsen in \cite{paper:Thomsen-htpyUHF} and later for AF-algebras $A$ by Nistor in \cite{paper:Nistor}. The key assumption that allows a complete computation is that $K_0(A)$ has large denominators (see \cite[Def.~2.2]{paper:Nistor}). This is automatic for simple infinite-dimensional $C^*$-algebras. Our results about the homotopy type of $\eqAut{\bT}{D \otimes \bK}$ can be seen as an equivariant generalisation of \cite{paper:Nistor}. Interestingly, the assumption of having large denominators might fail depending on the starting representation defining $D$. Nevertheless, we are able to determine all groups $\pi_n(\eqAut{\bT}{D \otimes \bK})$. Surprisingly, they are determined by the homotopy groups of $U(D^\bT)$, where $D^\bT$ is the fixed-point algebra, ie.\ the nonstable $K$-theory in the sense of \cite{paper:ThomsenNonstable} (see Prop.~\ref{prop:Aut_Proj_equivalence}, Prop.~\ref{prop:Proj_BU} and Thm.~\ref{thm:coefficients}): Let $p_V(t) = \sum_{i = 0}^d a_it^i \in \Z[t] \subset \Z[t,t^{-1}]$ be the character polynomial of the defining representation. Since $D$ is a direct limit of matrix $\bT$-algebras with connecting maps that induce multiplication by $p_V(t)$ in $\bT$-equivariant $K$-theory, we have $K_0^\bT(D) \cong \Z[t,t^{-1},p_V(t)^{-1}]$. Let
\begin{align*}
	\bddR &= \left\{ x \in K_0^\bT(D)\ | \ -m[1_D] \leq x \leq m[1_D] \text{ for some } m \in \N \right\} \\
	& \subset \Z[t,p_V(t)^{-1}] \quad \text{(see Cor.~\ref{cor:bddR_in_Zt})} \ ,\\
		\bddR^0 &= \{r \in \bddR \ |\ r(0) = 0 \} \ ,\\
	\bddR^\infty &= \left\{\frac{q}{p_V^k} \in \bddR \ |\ q \in \Z[t],\ k \geq 0,\ \deg(q) < kd \right\}\ .
\end{align*}
All homotopy groups in odd degrees vanish and the even ones are given by
\begin{align*}
	\pi_0(\eqAut{\bT}{D \otimes \bK}) &\cong GL_1(K_0^\bT(D)_+)\ , \\
	\pi_2(\eqAut{\bT}{D \otimes \bK}) &\cong \bddR\ .
\end{align*}
and for $k > 1$ we have $\pi_{2k}(\eqAut{\bT}{D \otimes \bK}) \cong \pi_{2k-1}(U(D^\bT))$ and therefore
\begin{align*}
	\pi_{2k}(\eqAut{\bT}{D \otimes \bK}) \cong
	\begin{cases}
		\bddR & \text{if } (a_0 > 1) \text{ and } (a_d > 1)\ , \\
		\bddR^0 & \text{if } (a_0 = 1) \text{ and } (a_d > 1)\ ,\\
		\bddR^\infty & \text{if } (a_0 > 1) \text{ and } (a_d = 1)\ ,\\
		\bddR^\infty \cap \bddR^0 & \text{if } (a_0 = 1) \text{ and } (a_d = 1)\ .
	\end{cases}
\end{align*}
In contrast to the non-equivariant case these groups are therefore in general only 2-periodic from degree 4 onwards. 

Feeding the coefficients $\check{E}^k_{D,\bT} = \pi_{-k}(\eqAut{\bT}{D \otimes \bK}$ into the Atiyah-Hirze\-bruch spectral sequence for $E^*_{D,\bT}(X)$ it is then possible to compute these groups for $X = \bT^n$. In this case the spectral sequence collapses at the $E_2$-page. The result is Cor.~\ref{cor:tori}. 

Finally, we take a closer look at the case where $D = \C$, ie.\ the case of the trivial representation. Even in this simplest case our results are interesting, since the group $E_{\C,\bT}^1(X)$ then comes with a natural homomorphism to the equivariant Brauer group $\text{Br}_{\bT}(X)$ for the trivial $\bT$-action on~$X$ (see \cite{paper:CrockerKumjianRaeburnWilliams} or \cite[Chap.~7]{book:RaeburnWilliams}). We show that this map is an isomorphism in Thm.~\ref{thm:BrauerGroup}. We will look into an interpretation of $E^1_{D,\bT}(X)$ as a generalisation of the equivariant Brauer group from \cite{paper:CrockerKumjianRaeburnWilliams} in future work. This should be compared to the non-equivariant case discussed in \cite{paper:DadarlatP-BrauerGroup}.

The guiding question for the programme initiated here is of course how the methods used in this paper transfer to other strongly self-absorbing $C^*$-dynamical systems. But there are several other lines of research that present themselves: While the lowest cohomological contribution to $E^1_{D,\bT}(X)$ can be understood via the edge homomorphism $E^1_{D,\bT}(X) \to H^1(X,GL_1(K_0^\bT(D)))$ in the Atiyah-Hirzebruch spectral sequence and has a bundle-theoretic interpretation, this is no longer true for higher cohomological data. For example, it would be interesting to see how a $\bT$-equivariant locally trivial $D \otimes \bK$-bundle $\mathcal{A} \to \bT^n$ can be constructed from classes in $H^{2n+1}(\bT^n,\bddR^{0,\infty})$ explicitly. 

Moreover, the cohomology theory $E^*_{D,\bT}(X)$ should be closely linked to equivariant stable homotopy theory. Evidence for this is the non-equivariant case \cite{paper:DadarlatP-UnitSpectra} and the fact that $E^1_{\C,\bT}(X)$ gives the $\bT$-equivariant third cohomology group of $X$ (with trivial $\bT$-action on $X$). We will explore the details of this in future work.\\[-2mm]

The article is structured as follows: In Sec.~\ref{sec:path-comp} we prove some preliminary results that are needed in later sections. We show that the action of $\bT$ on~$D$ is strongly self-absorbing (Lem.~\ref{lem:action_ssa}). We also take a closer look at the evaluation map at the projection $1 \otimes e \in (D \otimes \bK)^\bT$, where $e \in \bK$ denotes a rank $1$-projection onto the trivial representation throughout the article. It is then proven that the stabiliser $\eqAut{\bT,1\otimes e}{D \otimes \bK}$ is contractible, first in the case $D = \C$ in Thm.~\ref{thm:Aut_eT_K_contractible}, which is then used to prove the general case in Thm.~\ref{thm:Aut_eq_stabiliser_contractible}. The general argument is similar to the non-equivariant case. The section finishes with the first main result: the computation of the group of path-components $\pi_0(\eqAut{\bT}{D \otimes \bK})$ in Lem.~\ref{lem:pi0_equiv}.

In Sec.~\ref{sec:htpy-type} we consider the path-component of the identity in the equivariant automorphism group, denoted by $\eqAutId{\bT}{D \otimes \bK}$, and establish the homotopy equivalence $\eqAutId{\bT}{D \otimes \bK} \simeq BU(D^\bT)$. This is done in two steps: It is shown in Prop.~\ref{prop:Aut_Proj_equivalence} that $\eqAutId{\bT}{D \otimes \bK} \simeq \Proj{1 \otimes e}{(D \otimes \bK)^\bT}$ and in Prop.~\ref{prop:Proj_BU} that $\Proj{1 \otimes e}{(D \otimes \bK)^\bT} \simeq BU(D^\bT)$. The $K$-groups of $D^\bT$ are determined in Lem.~\ref{lem:K_groups_fixed_point_algebra}, which is then used to compute the homotopy groups of $U(D^\bT)$ in Thm.~\ref{thm:coefficients} giving the coefficients in Cor.~\ref{cor:coh_theory}. While the proof of the homotopy equivalence with $BU(D^\bT)$ follows the same lines as in the non-equivariant case in \cite{paper:DadarlatP-UnitSpectra} and \cite{paper:DadarlatP-DD-theory}, the computation of $\pi_n(U(D^\bT))$ is more intricate, since the fixed-point algebra is no longer simple and $K_0(D^\bT)$ may fail to have large denominators in the sense of \cite{paper:Nistor}.

The main classification result is contained in Sec.~\ref{sec:classification}. Similar to the non-equivariant case in \cite{paper:DadarlatP-UnitSpectra} the proof that $\eqAut{\bT}{D \otimes \bK}$ is an infinite loop space is based on diagram spaces called commutative $\cI$-monoids \cite{paper:Schlichtkrull}. We start Sec.~4 with a summary of the necessary background about them. It is then shown in Lem.~\ref{lem:stable_EHI-group} that $G^\bT_D(\mathbf{n}) = \eqAut{\bT}{(D\otimes \bK)^{\otimes n}}$ defines a commutative $\cI$-monoid in well-pointed topological groups that satisfies a compatibility condition between the group multiplication and the tensor product. The machinery developed in \cite{paper:DadarlatP-UnitSpectra} and outlined at the beginning of Sec.~\ref{sec:classification} then gives our main result Cor.~\ref{cor:coh_theory}. The computation of $E^1_{D,\bT}(\bT^n)$ is in Cor.~\ref{cor:tori} and the final comparison with the equivariant Brauer group is Thm.~\ref{thm:BrauerGroup}.

The appendix contains some elementary results about polynomials with non-negative integer coefficients that are needed in the computation of the groups $\pi_n(U(D^\bT))$ in Thm.~\ref{thm:coefficients}.



\section{The path-components of $\eqAut{\bT}{D \otimes \bK}$} \label{sec:path-comp}
Let $V$ be a finite-dimensional complex inner product space with a unitary $\bT$-action $\rho \colon \bT \to U(V)$ and define $D$ to be the $\bT$-algebra given by
\[
	D = \Endo{V}^{\otimes \infty}
\]
where $\bT$ acts by conjugation on each tensor factor. Let $\sigma = \Ad_{\rho^{\otimes \infty}}$ be this action. Consider the character subspaces of $V$ given by
\[
	V_k = \{\xi \in V \ | \ \rho(z)\xi = z^k\,\xi \quad \forall z \in \bT\}
\]
and note that $V$ decomposes into a finite direct sum
\begin{equation} \label{eqn:decomposition}
	V = \bigoplus_{k \in \Z} V_k\ .
\end{equation}
Let $\C_\ell$ be $\C$ as a vector space equipped with the $\bT$-action given by $z \cdot \xi = z^\ell \xi$ for $z \in \bT$. Let $D_\ell$ be the $\bT$-algebra given by
\[
	D_{\ell} = \Endo{\C_\ell \otimes V}^{\otimes \infty}
\]
and note that there is a canonical isomorphism $D \to D_{\ell}$ obtained as the infinite tensor product of the $*$-isomorphism $\Endo{V} \to \Endo{\C_{\ell} \otimes V}$ given by $T \mapsto \id{\C_{\ell}} \otimes T$. Taking the tensor product with $\C_{\ell}$ for an appropriate choice of $\ell$ we may without loss of generality assume that $V_k = 0$ for $k < 0$ and $V_0$ is the first non-trivial character subspace of $V$. 

\begin{definition}
	Given a $\bT$-representation $(V,\rho)$ a \emph{symmetry path} is a continuous path $\gamma \colon [0,1] \to U(V \otimes V)$ between the identity on $V \otimes V$ and the tensor flip $\xi \otimes \eta \mapsto \eta \otimes \xi$ that is $\bT$-equivariant in the sense that $[(\rho \otimes \rho)(z), \gamma(t)] = 0$ for all $t \in [0,1]$ and $z \in \bT$.
\end{definition}

\begin{lemma}
	Any $\bT$-representation $V$ has a symmetry path.
\end{lemma}

\begin{proof}
	The character subspaces of $W = V \otimes V$ are given by
	\[
		W_k = (V \otimes V)_k = \bigoplus_{i+j = k} V_i \otimes V_j
	\]
	It suffices to find continuous paths $\gamma_k \colon [0,1] \to U(W_k)$ between the identity and the tensor flip, since the direct sum $\gamma = \bigoplus \gamma_k$ will then have the desired properties. But the existence of $\gamma_k$ follows from the path-connectedness of the unitary group.
\end{proof}

\begin{lemma} \label{lem:action_ssa}
	There is a $\bT$-equivariant isomorphism $\varphi \colon D \to D \otimes D$ such that there are two continuous maps 
	\[
		u \colon [0,1) \to U(D \otimes D) \quad \text{and}\quad w \colon [0,1) \to U(D \otimes D) 
	\]
	with the properties $u(0) = 1 \otimes 1$, $(\sigma_z \otimes \sigma_z)(u(t)) = u(t)$, $w(0) = 1 \otimes 1$, $(\sigma_z \otimes \sigma_z)(w(t)) = w(t)$ for all $t \in [0,1)$ and $z \in \bT$ and 
	\begin{align*}
		\lim_{t \to 1} \lVert u(t)(d \otimes 1)u(t)^* - \varphi(d) \rVert &= 0\ , \\	
		\lim_{t \to 1} \lVert w(t)(1 \otimes d)w(t)^* - \varphi(d) \rVert &= 0\ .		
	\end{align*}
	In particular, the $C^*$-dynamical system $(D, \sigma)$ is strongly self-absorbing in the sense that the left tensor embedding $d \mapsto d \otimes 1$ is strongly asymptotically $\bT$-unitarily equivalent to an isomorphism.
\end{lemma}

\begin{proof}
	Consider the following two $*$-homomorphisms: 
	\[
		\varphi_{k,\ell} \colon \Endo{V}^{\otimes k} \otimes \Endo{V}^{\otimes \ell} \to \Endo{V}^{\otimes \max\{k,\ell\}} \otimes \Endo{V}^{\otimes \max\{k,\ell\}}\,
	\]
	maps $T \otimes S$ to $T \otimes 1 \otimes \dots \otimes 1 \otimes S$ in case $k \leq \ell$ with a tail of $\ell - k$-copies of $1 = 1_{\Endo{V}}$ and similarly for the second tensor factor in case $k > \ell$. The homomorphism 
	\[
		\kappa_{m} \colon \Endo{V}^{\otimes m} \otimes \Endo{V}^{\otimes m} \to \Endo{V}^{\otimes 2m}
	\]
	alternates the tensor factors, ie.\ $S_1 \otimes \dots S_m \otimes T_1 \otimes \dots \otimes T_m \mapsto S_1 \otimes T_1 \otimes \dots \otimes S_m \otimes T_m$. Let $\iota_m \colon \Endo{V}^{\otimes m} \to D$ be the inclusion map from the direct limit. The composition $\iota_{2\max\{k,\ell\}} \circ \kappa_{\max\{k,\ell\}} \circ \varphi_{k,l}$ factors through the limit and gives a $\bT$-equivariant isomorphism
	\[
		\varphi^{-1} \colon D \otimes D \to D\ .
	\]
	To find $u$ it now suffices to construct a continuous map $v \colon [0,\infty) \to U(D)$ such that $\sigma_z(v(t)) = v(t)$ for all $t \in [0,\infty)$ and  
	\begin{equation} \label{eqn:asymptotic}	
		\lim_{t \to \infty} \lVert v(t)\varphi^{-1}(d \otimes 1)v(t)^* - d \rVert = 0\ .
	\end{equation}
	For $d = T_1 \otimes T_2 \otimes T_3 \otimes \dots$ with $T_i \in \Endo{V}$ the image $\varphi^{-1}(d \otimes 1)$ will take the form
	\[
		T_1 \otimes 1 \otimes T_2 \otimes 1 \otimes T_3 \otimes \dots 
	\] 
	Choose a symmetry path $\gamma \colon [0,1] \to U(V \otimes V)$ and let $u_i \colon [0,1] \to U(D)$ be defined by
	\[
		u_i(t) = \underbrace{1 \otimes \dots \otimes 1}_{(i-1)\ \text{factors}} \otimes \gamma(t) \otimes 1 \otimes \dots 
	\]
	Each $u_i$ satisfies $\sigma_z(u_i(t)) = u_i(t)$ for all $t$. 
	Note that conjugation of $d$ by $u_2$ gives a continuous path between 
	\[
		T_1 \otimes 1 \otimes T_2 \otimes 1 \otimes T_3 \otimes \dots 
	\qquad 
	\text{and} 
	\qquad
		T_1 \otimes T_2 \otimes 1 \otimes 1 \otimes T_3 \otimes \dots 
	\] 
	Conjugation of $d$ by $u_3(t-1)u_4(t-1)u_2(1)$ for $t \in [1,2]$ then gives a path between
	\[
		T_1 \otimes T_2 \otimes 1 \otimes 1 \otimes T_3 \otimes \dots 	
		\qquad
		\text{and}
		\qquad
		T_1 \otimes T_2 \otimes T_3 \otimes 1 \otimes 1 \otimes \dots 			
	\]
	Iterating this procedure we define $v$ to be $u_2$ on the interval $[0,1]$ and on $(k,k+1]$ we set
	\[
		v(t) = u_{k+2}(t-k)\,\dots\,u_{2k+2}(t-k)\,v(k)\ ,
	\]
	which is point-wise a unitary in the fixed-point algebra. Let $D_0$ be the dense subalgebra defined by
	\[
		D_0 = \bigcup_{k \in \N} \Endo{V}^{\otimes k} \subset D\ .
	\]
	For $d \in D_0$ conjugation of $\varphi^{-1}(d \otimes 1)$ by the path $v$ will reorder the tensor factors along the path until we end up with the element $d$ itself. Therefore \eqref{eqn:asymptotic} holds by construction. An $\frac{\varepsilon}{3}$-argument then shows that \eqref{eqn:asymptotic} is true for all $d \in D$. This finishes the construction of $u$. In a similar way the path $w$ can be constructed from an asymptotic unitary equivalence between $\varphi^{-1}(1 \otimes d)$ and the identity permuting the tensor factors of $1 \otimes T_1 \otimes 1 \otimes T_2 \otimes \dots$.
\end{proof}

Next we will have a closer look at the compact operators $\bK$ as a $\bT$-algebra with respect to an action that has every character appearing with infinite multiplicity. To this end let $H_0$ be an infinite-dimensional separable Hilbert space and let $H = \ell^2(\Z) \otimes H_0$. Let $\{\delta_i\}_{i \in \Z}$  be the standard orthonormal basis of $\ell^2(\Z)$. In the following we will identify $H_0$ with $\{\delta_0 \otimes \eta \ | \ \eta \in H_0\}$. For $z \in \bT$ define 
\begin{equation} \label{eqn:action_on_K}
	U_z \colon H \to H \quad ; \quad U_z(\delta_k \otimes \eta) = z^k\,\delta_k \otimes \eta\ .	
\end{equation}
As alluded to, this gives a unitary representation of $\bT$ on $H$ that contains each irreducible representation with infinite multiplicity. Let 
\[
	H_k = \{ \delta_k \otimes \eta \ | \ \eta \in H_0 \}\ ,
\]
which is the character subspace of $H$ corresponding to $\chi(z) = z^k$. Let $\bK = \bK(H)$ be the compact operators on $H$. Define $\pi(z) = \Ad_{U_z}$ to be the representation of $\bT$ on $\bK$ corresponding to $U$ and let $\bK^\bT$ be the fixed-point algebra of $\bK$. 

An operator $T \in \bK$ lies in $\bK^\bT$ if and only if it commutes with $U_z$ for all $z \in \bT$, which is the case if and only if it maps each $H_k$ to itself.  Therefore 
\[
\bK^\bT \cong \bigoplus_{k \in \Z} \bK(H_0) \cong C_0(\Z) \otimes \bK(H_0)\ .
\] 
In fact, $\bK$ is equivariantly isomorphic to a crossed product $\bK^\bT \rtimes \Z$. To see this let $T = [T_i]_{i \in \Z} \in \bK^\bT$ and define $\rho \colon \Z \to \Aut{\bK^\bT}$ by $[\rho(k)(T)]_j = T_{j+k}$. The map
\begin{equation} \label{eqn:crossed_prod_iso}
	\bK^\bT \rtimes_{\rho} \Z \to \bK
\end{equation}
which is the inclusion $\bK^\bT \to \bK$ on the fixed-point algebra and sends $n \in \Z$ to the unitary operator
\[
	W_n \colon H \to H \quad ; \quad W_n(\delta_{k} \otimes \eta) = \delta_{k-n} \otimes \eta\ .
\]
is an isomorphism. To see why, let $T \colon H_k \to H_{k+n}$ be any compact operator that is $0$ on the complement of $H_k$. Then we have $TW_n \in \bK^\bT$ and $(TW_n, -n) \in \bK^\bT \rtimes \Z$ is mapped to $T \in \bK$. These operators span a dense subalgebra, which proves that \eqref{eqn:crossed_prod_iso} is surjective. Injectivity follows similarly by considering the restrictions to the ``matrix entries'' $H_k \to H_{k+n}$. The isomorphism \eqref{eqn:crossed_prod_iso} is an instance of Takai duality, since
\[
\bK^\bT \rtimes_{\rho} \Z \cong (C_0(\Z) \rtimes \Z) \otimes \bK(H_0) \cong \bK(\ell^2(\Z) \otimes H_0) = \bK
\]
and this isomorphism intertwines $\pi$ with dual action of $\bT$ on $\bK^\bT \rtimes \Z$. 

Lastly, we observe the following property of the equivariant automorphisms of~$\bK$. Suppose that $\alpha = \Ad_{W} \in \eqAut{\bT}{\bK}$. This means that $\alpha \circ \pi(z) = \pi(z) \circ \alpha$ for all $z \in \bT$, which implies
\begin{equation} \label{eqn:equiv_unitary}
	WU_zTU_z^*W^* = U_zWTW^*U_z^* \quad \Rightarrow \quad [U_z^*W^*U_zW,T] = 0
\end{equation}
for all $T \in \bK$. Therefore $U_zW = \chi(z)\,WU_z$ with $\chi(z) \in U(1)$ and it is straightforward to check that $\chi \colon \bT \to U(1)$ is a character, ie.\ $\chi(z) = z^k$ for some $k \in \Z$. In light of this we define $U^\bT(H) \subset U(H)$ to be the subgroup given by
\[
	U^\bT(H) = \{ W \in U(H) \ | \ U_zW = \chi(z)\,WU_z \text{ for some character } \chi \text{ of } \bT\}\ .
\] 
Since the Pontryagin dual $\hat{\bT}$ is isomorphic to $\Z$, the above gives a group homomorphism 
\begin{equation} \label{eqn:pi0-hom_for_K}
	\theta \colon \eqAut{\bT}{\bK} \to \Z \ .
\end{equation}
We will see in Sec.~\ref{sec:BrauerGroup} how $\theta$ gives rise to the Phillips-Raeburn obstruction for $\bT$-equivariant bundles \cite{paper:PhillipsRaeburn-locunitary}. Note that $\theta(\Ad_W) = n$ implies that $W$ restricts to a unitary isomorphism $H_k \to H_{n+k}$ for each $k \in \Z$. As we will see, $\theta$ really has to be understood as a homomorphism $\theta \colon \eqAut{\bT}{\bK} \to GL_1(R(\bT))$, and we will generalise it to $\eqAut{\bT}{D \otimes \bK}$ for an infinite UHF-algebra $D$. 

\begin{theorem} \label{thm:Aut_eT_K_contractible}
	Let $H$ be as above and let $e \in \bK$ be a rank $1$-projection onto a subspace of $H_0$. Denote by $\eqAut{\bT,e}{\bK}$ the stabiliser of $e$ in $\eqAut{\bT}{\bK}$. We have a homeomorphism 
	\[
		\eqAut{\bT,e}{\bK} \cong \{ W \in U^{\bT}(H) \ | \left.W\right|_{eH} = 1_{eH} \} =: G
	\]
	where $G$ is equipped with the strong topology. Moreover, $G$ is contractible.
\end{theorem}

\begin{proof}
	By \cite[Prop.~8.1]{book:Lance} the strict topology on $U(H)$ as the unitary group of the multiplier algebra $M(\bK)$ is the same as the strong topology. Therefore the natural map $G \to \eqAut{\bT,e}{\bK}$ given by $W \mapsto \Ad_W$ is continuous with respect to the point-norm topology on $\eqAut{\bT,e}{\bK}$ and the strong topology on~$G$. It suffices to see that it has a continuous inverse. Let $W \in U(H)$ with $\Ad_W \in \eqAut{\bT,e}{\bK}$	and note that $eWe = eW = We = \lambda_We$ for some $\lambda_W \in U(1)$. Consider the map
	\[
		\eqAut{\bT,e}{\bK} \to G \quad ; \quad \Ad_W \mapsto \lambda_W^*W\ .
	\]
	This is well-defined, since $\lambda_W^*W$ does not change when $W$ is multiplied by an element in $U(1)$.  Moreover, $\left.\lambda_W^*W\right|_{eH} = 1$ holds by definition and $\lambda_W^*W \in U^{\bT}(H)$ by the observations made in \eqref{eqn:equiv_unitary}. 
	
	To check continuity let $U_e(H) \subset U^\bT(H)$ be the subgroup of operators commuting with $e$ and let $\eta_0 \in eH$ be a unit vector. The map $U_e(H) \to \C$ defined by $W \mapsto \lambda_W$ is continuous with respect to the strong topology on the domain, since $W\eta_0 = \lambda_W\eta_0$. Hence, $U_e(H) \to G$ given by $W \mapsto \lambda_W^*W$ is also strongly continuous. By \cite[Prop.~1.6]{book:RaeburnWilliams} conjugation provides a homeomorphism $\Aut{\bK} \cong PU(H)$, where the right hand side is equipped with the quotient of the strong topology. Combining these two statements shows that $\eqAut{\bT,e}{\bK} \to G$ given by $\Ad_W \mapsto \lambda_W^* W$ is continuous. It is clearly the inverse of $G \to \eqAut{\bT,e}{\bK},\ W \mapsto \Ad_W$.
	
	It remains to show that $G$ is contractible. Let $W \in G$. Since $e$ projects to a subspace of $H_0$ we have $U_z e = e$ and therefore
	\[
		\eta_0 = U_z W\eta_0 = \chi(z)\,WU_z\eta_0 = \chi(z) \eta_0
	\]
	for all $z \in \bT$ and $\eta_0 \in eH$. This implies that $\chi$ is trivial and therefore $\left.W\right|_{H_k} \colon H_k \to H_k$. Thus, every element in $G$ is a block sum of unitaries, more precisely
	\[
		G \cong U((1-e)H_0) \times \prod_{k \in \Z \atop k \neq 0} U(H_0)
	\]
	where the right hand side carries the product topology. Now the statement follows from the contractibility of the unitary group of a separable Hilbert space in the strong topology.
\end{proof}

\begin{lemma} \label{lem:path_in_K}
	Let $H$ be as above and let $e \in \bK$ be a rank $1$-projection onto a subspace of $H_0$. The group $\bT$ acts on $\bK \otimes \bK$ diagonally. There exists a point-norm continuous path 
	\[
		\gamma  \colon [0,1] \to \hom_{\bT}(\bK, \bK \otimes \bK)
	\]
	into the $\bT$-equivariant homomorphisms with the following properties: 
	\begin{enumerate}[(a)]
		\item $\gamma(0)(T) = T \otimes e$ and $\gamma(1)(T) = e \otimes T$,
		\item $\gamma$ restricts to $(0,1) \to \eqIsom{\bT}{\bK}{\bK \otimes \bK}$ 
		\item $\gamma(t)(e) = e \otimes e$ for all $t \in [0,1]$.
	\end{enumerate}
\end{lemma}

\begin{proof}
	The construction of such a path in the non-equivariant setting is described in the proof of \cite[Thm.~2.5]{paper:DadarlatP-DD-theory}. The idea in the equivariant case is to obtain a path on fixed-point algebras using the non-equivariant result and then extend it to $\bK \otimes \bK$ by the crossed product decomposition.

	Denote the diagonal action of $\bT$ on $\bK \otimes \bK$ by $\Delta$ and the corresponding fixed-point algebra by $(\bK \otimes \bK)^\Delta$. Let $\hat{H}_m \subset H \otimes H$ be the character subspace of $\Delta$ for the character $z \mapsto z^m$, ie.\  
	\[
		\hat{H}_m = \bigoplus_{k + l = m} H_k \otimes H_l \ .
	\]
	Since $T \in (\bK \otimes \bK)^{\Delta}$ if and only if $T$ maps each $\hat{H}_m$ to itself, we see that 
	\begin{equation} \label{eqn:equiv_isos_compact}
		(\bK \otimes \bK)^{\Delta} \cong \bigoplus_{m \in \Z} \bK(\hat{H}_m) \cong \bigoplus_{m \in \Z} \bK(\hat{H}_0) \cong C_0(\Z) \otimes \bK(\hat{H}_0)\ .
	\end{equation}
	While the first and the last isomorphism in this line are canonical, there are two natural choices for the second one given either by the conjugation with $W_m \otimes \id{}$ or with $\id{} \otimes W_m$ (ie.\ shifting the first index in the tensor product to identify $\hat{H}_m$ with $\hat{H}_0$ or the second one). Each option induces a $\bT$-equivariant isomorphism 
	\[
		(\bK \otimes \bK)^{\Delta} \rtimes \Z \to \bK \otimes \bK
	\] 
	and we will denote these by $\theta_l$ and $\theta_r$ for $W_m \otimes \id{}$ and $\id{} \otimes W_m$ respectively. Both of them intertwine the dual action on the left hand side with the diagonal action on the right. Choose an isomorphism $\psi_1 \colon \bK(H_0) \to \bK(\hat{H}_0)$ such that $\psi_1(e) = e \otimes e \in \bK(H_0 \otimes H_0) \subset \bK(\hat{H}_0)$. As explained in the proof of \cite[Thm.~2.5]{paper:DadarlatP-DD-theory} there is a continuous path 
	\[
		\gamma_1 \colon [0,1] \to \hom(\bK(H_0), \bK(\hat{H}_0))
	\] 
	with the properties
	\begin{enumerate}[(a)]
		\item $\gamma_1(0)(T) = T \otimes e$, $\gamma_1(1)(T) = \psi_1(T)$,
		\item $\gamma_1$ restricts to $(0,1] \to \text{Iso}(\bK(H_0), \bK(\hat{H}_0))$ 
		\item $\gamma_1(t)(e) = e \otimes e$ for all $t \in [0,1]$.
	\end{enumerate}
	Note that $\id{C_0(\Z)} \otimes \gamma_1$ takes values in $\hom(C_0(\Z) \otimes \bK(H_0), C_0(\Z) \otimes \bK(\hat{H}_0))$. Using the isomorphism \eqref{eqn:equiv_isos_compact} with $W_m \otimes \id{}$ we can extend this to a continuous path in $\Z$-equivariant $*$-homomorphisms of the form
	\[
		\bar{\gamma}_1 \colon [0,1] \to \hom_{\Z}(\bK^\bT, (\bK \otimes \bK)^\Delta)\ .
	\]
	In particular, by our choice of $\Z$-action on $(\bK \otimes \bK)^\Delta$, we have that $\bar{\gamma}_1(0)(T) = T \otimes e$ is $\Z$-equivariant and $\bar{\gamma}_1(t)(e) = e \otimes e$ for $t \in [0,1]$. Forming the crossed product with the $\Z$-action pointwise and applying $\theta_l$ results in  
	\[
		\hat{\gamma}_1 \colon [0,1] \to \hom_{\bT}(\bK, \bK \otimes \bK)\ .
	\]
	By construction $\hat{\gamma}_1$ has the property that $\hat{\gamma}_1(0)(T) = T \otimes e$, it is an isomorphism for all $t \in (0,1]$ and $\hat{\gamma}_1(t)(e) = e \otimes e$ for all $t \in [0,1]$.

	Repeating this construction with $T \mapsto e \otimes T$, using $\id{} \otimes W_m$ instead of $W_m \otimes \id{}$ gives another continuous path $\hat{\gamma}_2$ with the analogous properties that ends in a potentially different isomorphism $\hat{\gamma}_2(1)$ with $\hat{\gamma}_1(t)(e) = e \otimes e$ for all $t \in [0,1]$. But we have seen in Theorem~\ref{thm:Aut_eT_K_contractible} that $\eqAut{\bT,e}{\bK}$ is contractible, so in particular path-connected. Hence, there is a continuous path running through $\eqAut{\bT}{\bK}$ connecting $\hat{\gamma}_1(1)$ and $\hat{\gamma}_2(1)$. Concatenating these three paths appropriately gives the result.
\end{proof}

By Lem.~\ref{lem:action_ssa} there is a $\bT$-equivariant isomorphism $\varphi \colon D \to D \otimes D$ and a continuous path $u \colon [0,1) \to U(D \otimes D)$ with image in the $\bT$-invariant unitaries such that 
	\[
		\lim_{t \to 1} \lVert (d \otimes 1) - u(t)^*\varphi(d)u(t) \rVert = 0
	\]
	and $u(0) = 1 \otimes 1$. This gives a continuous path
	\[
		\gamma_l \colon [0,1] \to \hom_{\bT}(D, D \otimes D) \quad ; \quad t \mapsto \begin{cases}
			d \mapsto u(t)^*\varphi(d)u(t) & \text{for } t \in [0,1) \\
			d \mapsto d \otimes 1	& \text{for } t = 1
		\end{cases}
	\]
	between $\varphi$ and the left tensor embedding $d \mapsto d \otimes 1$, where the equivariant homomorphisms are equipped with the pointwise-norm topology. Likewise, there is a similar path $\gamma_r$, constructed using $w$ from Lem.~\ref{lem:action_ssa}, that starts at~$\varphi$ and ends in the right tensor embedding $d \mapsto 1 \otimes d$. Concatenating $\gamma_l$ and $\gamma_r$ (and reversing one of them) we end up with 
	\begin{equation} \label{eqn:path_gamma}
		\gamma \colon [0,1] \to \hom_{\bT}(D, D \otimes D)
	\end{equation}
	with the properties that $\gamma(0)(d) = d \otimes 1$, $\gamma(1)(d) = 1 \otimes d$ and $\gamma(t)$ is an isomorphism for $t \in (0,1)$ with $\gamma(\tfrac{1}{2}) = \varphi$. 
	
Let $\delta \colon [0,1] \to \hom_{\bT}(\bK, \bK \otimes \bK)$ be a continuous path from Lem.~\ref{lem:path_in_K}. Let $\tau \colon D \otimes D \otimes \bK \otimes \bK \to D \otimes \bK \otimes D \otimes \bK$ be the map interchanging the middle two tensor factors. The continuous path
	\begin{equation} \label{eqn:path_beta}
		\beta \colon [0,1] \to \hom_{\bT}(D \otimes \bK, D \otimes \bK \otimes D \otimes \bK)
	\end{equation} 
	given by $\tau \circ (\gamma \otimes \delta)$ satisfies $\beta(0)(a) = a \otimes (1 \otimes e)$, $\beta(1)(a) = (1 \otimes e) \otimes a$ and $\beta(t)$ is an isomorphism for $t \in (0,1)$.	 Let $\psi = \beta(\tfrac{1}{2})$. Using $\beta$ we now construct the following homotopy
\begin{align} \label{eqn:key_homotopy}
	H_l &\colon \eqAut{\bT}{D \otimes \bK} \times [0,1] \to \eqAut{\bT}{D \otimes \bK} \notag \\
	H_l(\alpha,t) &= \begin{cases}
		\beta(\tfrac{1}{2}t)^{-1} \circ (\alpha \otimes \id{}) \circ \beta(\tfrac{1}{2}t) & \text{for } t \in (0,1] \ ,\\
		\alpha & \text{for } t= 0\ .
	\end{cases}
\end{align}
between the identity map on $\eqAut{\bT}{D \otimes \bK}$ and $\alpha \mapsto \psi^{-1} \circ (\alpha \otimes \id{}) \circ \psi$. Using the other half of the unit interval and $\id{} \otimes \alpha$ instead of $\alpha \otimes \id{}$ we can also construct a homotopy $H_r$ between the identity and $\alpha \mapsto \psi^{-1} \circ (\id{} \otimes \alpha) \circ \psi$. The paths $\gamma$ and $\beta$ and the homotopies $H_l$ and $H_r$ will play a crucial role in the proofs of Thm.~\ref{thm:Aut_eq_stabiliser_contractible}, Lem.~\ref{lem:pi0_equiv} and Lem.~\ref{lem:stable_EHI-group}.

\begin{theorem} \label{thm:Aut_eq_stabiliser_contractible}
	Let $D = \Endo{V}^{\otimes \infty}$ the $\bT$-algebra arising from a $\bT$-re\-pre\-sen\-ta\-tion $V$ and let $e \in \bK(H_0)$ be a rank 1-projection. 	The topological group $\eqAut{\bT, 1 \otimes e}{D \otimes \bK}$ of $\bT$-equivariant automorphisms fixing the element $1 \otimes e$ is contractible. Likewise, the group $\eqAut{\bT}{D}$ of unital $\bT$-equivariant automorphisms of $D$ is contractible.
\end{theorem}

\begin{proof}
	Let $\gamma$ be as in \eqref{eqn:path_gamma}. As in the non-equivariant case (see \cite[Thm.~2.3 and Thm.~2.5]{paper:DadarlatP-DD-theory}) the homotopy
\begin{align*} 
	H &\colon \eqAut{\bT}{D} \times [0,1] \to \eqAut{\bT}{D} \notag \\
	H(\alpha,t) &= \begin{cases}
		\id{D} & \text{for } t= 1\ ,\\
		\gamma(t)^{-1} \circ (\alpha \otimes \id{D}) \circ \gamma(t) & \text{for } t \in (0,1) \ ,\\
		\alpha & \text{for } t= 0\ .
	\end{cases}
\end{align*}
	provides a contraction of $\eqAut{\bT}{D}$ onto $\id{D}$ (compare with \cite[Lem.~2.2]{paper:DadarlatP-DD-theory}). 
	
	Replacing $\eqAut{\bT}{D}$ by the stabiliser $\eqAut{\bT, 1 \otimes e}{D \otimes \bK}$ and using $\beta$ as in~\eqref{eqn:path_beta} instead of $\gamma$ in the definition of $H$ gives a contraction of $\eqAut{\bT, 1 \otimes e}{D \otimes \bK}$. Note in particular that $\beta(t)(1 \otimes e) = (1 \otimes e)^{\otimes 2}$ for all $t \in [0,1]$.
\end{proof}

The following simple algebraic observation is the key to understanding the invertible elements in localisations of the representation ring $\Rep{\bT}$. 

\begin{lemma} \label{lem:units_in_loc}
	Let $R$ be a unique factorization domain and let $S \subset R$ be a multiplicative subset with $0 \notin S$. If $u \in R[S^{-1}]$ is a unit, then 
	\[
		u = v\,p_1^{k_1}\dots p_r^{k_r}
	\]
	where $v \in R$ is a unit, $p_j$ are primes of elements in $S$ and $k_j \in \Z$.
\end{lemma}

\begin{proof}
	Without loss of generality we may assume that if $rs \in S$, then $r$ and~$s$ are both elements of $S$, since adding these elements to $S$ does not change the localisation\footnote{A multiplicative subset $S$ with that property is usually called saturated, but we will avoid this terminology, since it clashes with saturated group actions.}. In particular, any prime factor of an element in $S$ is itself in $S$. 	Note that if $p \in R$ is any element such that $p \in R[S^{-1}]$ is a unit, then $p \in S$. Indeed, there exist $x \in R$ and $s \in S$ with
	\[		
		p \frac{x}{s} = 1 \quad \Leftrightarrow \quad p\,x = s \quad \Rightarrow \quad p \in S 
	\]
	by our assumption about $S$. Now let $u \in R[S^{-1}]$ be a unit. There exists $s \in S$ with $u\cdot s \in R$. Let 
	\[
		u \cdot s = v\,q_1^{\ell_1}\dots q_{m}^{\ell_m}
	\]
	be the prime factorization of $u \cdot s$ with a unit $v \in R$, prime factors $q_i \in R$ and $\ell_i \in \N$. Since $u$ and $s$ are both invertible in the localization $R[S^{-1}]$, each of the $q_i \in R$ is also invertible in $R[S^{-1}]$. By our previous observation we have $q_i \in S$. But as we noted above, the prime factors of $s$ are in $S$ as well and the result follows. 
\end{proof}

Given a $\bT$-algebra $A$ there is a natural homomorphism $K_0(A^\bT) \to K_0^\bT(A)$. In case $A$ is unital it is defined as follows: It maps the $K$-theory class of a projection $p \in M_k(A^\bT)$ to the class of the $\bT$-equivariant right Hilbert $A$-module $p\,\C^k \otimes A$ (see \cite[Prop.~7.1.6]{book:Phillips}). In general the map on $K$-theory is induced by a $*$-homomorphism $j \colon A^\bT \to A \rtimes \bT$ constructed in \cite[Lem.~7.1.7]{book:Phillips}. Moreover, recall that a $\bT$-action $\alpha$ on a $C^*$-algebra $A$ is \emph{saturated} if the completion $\overline{A}$ of $A$ with respect to the $A^\bT$-valued inner product 
\[
	\lscal{A^\bT}{x}{y} = \int_\bT \alpha_z(xy^*)\,dz\ ,
\] 
left $A^\bT$-multiplication inherited from $A$ and right multiplication given by $x f = \int_\bT \alpha_z^{-1}(xf(z))\,dz$ for $f \in L^1(\bT,A)$ is an $A^\bT$-$A \rtimes \bT$-equivalence bimodule (see \cite[Def.~7.1.4]{book:Phillips}).

\begin{lemma} \label{lem:saturated}
Let $D = \Endo{V}^{\otimes \infty}$ be the $\bT$-algebra arising from a $\bT$-representation $V$.
\begin{enumerate}[i)]
	\item \label{it:sat_i} The $\bT$-action on $D \otimes \bK$ is saturated.
	\item \label{it:sat_ii} The homomorphism $K_0((D \otimes \bK)^{\bT}) \to K_0^{\bT}(D \otimes \bK)$ is an isomorphism.
\end{enumerate}
\end{lemma}

\begin{proof}
	We first claim that the $\bT$-action on $\bK$ is saturated. This follows from an application of Rieffel's criterion (see \cite[Thm.~7.1.15]{book:Phillips}). Indeed, for $m \in \Z$ we have 
	\[
		\bK_m = \{ T \in \bK \ | \ \pi(z)(T) = z^m\,T \} = \{ T \in \bK \ | \ \left.T\right|_{H_k} \colon H_k \to H_{k+m}\ \forall k \in \Z\}.
	\]	
	and therefore $\overline{\bK_m^*\bK_m} = \overline{\bK_{-m}\bK_m} = \bK_0 = \bK^\bT$ for all $m \in \Z$. Next note that $\Endo{V}^{\otimes k} \otimes \bK \cong_{\bT} \bK$, since $H = \ell^2(\Z) \otimes H_0$ contains all irreducible $\bT$-representations with infinite multiplicity. Hence, by \cite[Prop.~7.1.13]{book:Phillips} the first statement follows. 
	
 The natural homomorphism $K_0((D \otimes \bK)^\bT) \to K_0^\bT(D \otimes \bK)$ is in fact an isomorphism by \eqref{it:sat_i} combined with \cite[Prop.~7.1.8]{book:Phillips}.
\end{proof}

Let $e$ be a fixed invariant rank~$1$-projection in $\bK$ and consider the stabilisation homomorphism $D \to D \otimes \bK$ with $d \mapsto d \otimes e$. By \cite[Lem.~2.7.5]{book:Phillips} this induces an isomorphism $K_0^\bT(D) \to K_0^\bT(D \otimes \bK)$. We will see that this isomorphism and the one in Lem.~\ref{lem:saturated} \eqref{it:sat_ii} are both ring homomorphisms. The ring structure on $K_0^\bT(D)$ is given by
\[
	K_0^\bT(D) \times K_0^\bT(D) \to K_0^{\bT}(D \otimes D) \to K_0^\bT(D)\ ,
\]
where the first map is the tensor product pairing (see \cite[Prop.~6.1.3]{book:Phillips}) and the second map is induced by an equivariant isomorphism $D \otimes D \to D$ which exists since the action is strongly self-absorbing. Likewise, there is an analogous map
\[
	K_0^\bT(D \otimes \bK) \times K_0^\bT(D \otimes \bK) \to K_0^\bT(D \otimes \bK)
\]
using the tensor product and an equivariant isomorphism $(D \otimes \bK)^{\otimes 2} \to D \otimes \bK$ obtained as a tensor product of two equivariant isomorphisms $D \otimes D \to D$ and $\bK \otimes \bK \to \bK$, where we demand that the latter maps $e \otimes e$ to $e$. A similar construction works for the fixed-point algebras. We obtain a homomorphism  
\[
	K_0((D \otimes \bK)^{\bT}) \times K_0((D \otimes \bK)^{\bT}) \to K_0((D \otimes \bK)^{\bT} \otimes (D \otimes \bK)^{\bT}) \to K_0( (D \otimes \bK)^\bT )
\]
where the first map is again given by the tensor product. The second map is induced by the embedding of $(D \otimes \bK)^{\bT} \otimes (D \otimes \bK)^{\bT}$ into $( D \otimes \bK \otimes D \otimes \bK )^{\bT}$ followed by the equivariant isomorphism as above. 

Now consider the following diagram 
\[
\begin{tikzcd}
	K_0^\bT(D) \times K_0^\bT(D) \ar[d,"\cong" left] \ar[r] & K_0^\bT(D) \ar[d,"\cong"] \\
	K_0^\bT(D \otimes \bK) \times K_0^\bT(D \otimes \bK) \ar[r] & K_0^\bT(D \otimes \bK) \\
	K_0((D \otimes \bK)^\bT) \times K_0((D \otimes \bK)^\bT) \ar[r] \ar[u,"\cong"] & K_0((D \otimes \bK)^\bT) \ar[u,"\cong" right]
\end{tikzcd}
\]
in which the upper vertical arrows are induced by the stabilisation homomorphism. The upper square commutes by the naturality of the tensor product pairing and our choice of isomorphisms. The middle horizontal map is induced by the tensor product pairing for $D \otimes \bK$ in equivariant $K$-theory followed by isomorphisms as described above. Both of the tensor product pairings used in the middle and the lower arrow can be obtained by unitising the algebra, and it is straightforward to check that this implies that the lower square also commutes. In particular, the composition $K_0((D \otimes \bK)^\bT) \to K_0^\bT(D)$ of the homomorphism from Lem.~\ref{lem:saturated} \eqref{it:sat_ii} with the inverse of the stabilisation map is a ring homomorphism.

By stability we have $K_0^\bT(\Endo{V}^{\otimes k}) \cong K_0^\bT(\C) \cong R(\bT) \cong \Z[t,t^{-1}]$. These identifications become isomorphisms of ordered groups if we define the positive cone $R(\bT)_+$ of $R(\bT)$ to be equivalence classes of representations (as opposed to formal differences) and the one of $\Z[t,t^{-1}]$ to be polynomials with non-negative coefficients. The continuity of equivariant $K$-theory implies $K_0^\bT(D) \cong \Z[t,t^{-1},p_V(t)^{-1}]$. Indeed, the map
\[
	K_0^\bT(\Endo{V}^{\otimes k}) \to K_0^\bT(\Endo{V}^{\otimes (k+1)})
\]
corresponds to multiplication with $p_V(t)$ after identifying both sides with $\Z[t,t^{-1}]$. Under the above isomorphism the positive cone $K_0^\bT(D)_+ \subset K_0^\bT(D)$ is mapped onto the semiring $\N_0[t,t^{-1},p_V(t)^{-1}]$ consisting of polynomials $q/p_V^k$ where $q$ has non-negative coefficients (note that the coefficients of $p_V$ are always non-negative). It also follows from the description as a direct limit that $K_0^\bT(D) \cong \Z[t,t^{-1},p_V(t)^{-1}]$ is compatible with the multiplicative structure on both sides and therefore an isomorphism of rings. Moreover, the map $K_0((D \otimes \bK)^\bT) \to K_0^\bT(D)$ defined above identifies the positive cone $K_0^\bT(D)_+$ with
\(
	 \{ [p]_0 \in K_0((D \otimes \bK)^\bT) \ |\ p \in (D \otimes \bK)^\bT \}
\). 

For a semiring $R_+ \subset R$ contained in a commutative unital ring $R$, define 
\[
	GL_1(R_+) = GL_1(R) \cap R_+\ ,
\]
ie.\ those elements in $R_+$ that are invertible when considered as elements of~$R$. 

\begin{lemma} \label{lem:pi0_equiv}
	Let $D = \Endo{V}^{\otimes \infty}$ be the $\bT$-algebra arising from a $\bT$-representation $V$. We have a group isomorphism
	\[
		\pi_0(\eqAut{\bT}{D \otimes \bK}) \cong GL_1(K_0^{\bT}(D)_+) \cong GL_1(\N_0[t,t^{-1},p_V(t)^{-1}]) 
	\]
\end{lemma}

\begin{proof}
	Let $e \in \bK(H_0) \subset \bK$ be a rank $1$-projection. For $\alpha \in \eqAut{\bT}{D \otimes \bK}$ we have $\alpha(1 \otimes e) \in (D \otimes \bK)^{\bT}$. This gives a map
	\begin{equation} \label{eqn:pi0_hom}
		\kappa \colon \pi_0(\eqAut{\bT}{D \otimes \bK}) \to K_0((D\otimes \bK)^{\bT})_+ \cong K_0^{\bT}(D)_+\ ,
	\end{equation}
	where the isomorphism is the one from Lem.~\ref{lem:saturated} \eqref{it:sat_ii}. Apart from the multiplication induced by composition, the group $\pi_0(\eqAut{\bT}{D \otimes \bK})$ carries a second multiplication $\ast$ induced by the tensor product of automorphisms: Let $\beta \colon [0,1] \to \hom_{\bT}(D \otimes \bK, D \otimes \bK \otimes D \otimes \bK)$ be the path from~\eqref{eqn:path_beta}. Let $\psi = \beta(\tfrac{1}{2})$ and define
	\[
		[\alpha_1] \ast [\alpha_2] = [\psi^{-1} \circ (\alpha_1 \otimes \alpha_2) \circ \psi]
	\]
	(compare with \cite[Lem.~2.13]{paper:DadarlatP-DD-theory}).
	
	Let $H_l$ be as in \eqref{eqn:key_homotopy}. The continuous path $H_{\alpha} \colon [0,1] \to \eqAut{\bT}{D \otimes \bK}$ given by $H_{\alpha}(t) = H_l(\alpha,t)$ shows that $[\alpha] \ast [\id{}] = [\alpha]$. Using $H_r$ instead of $H_l$ we see that $[\id{}] \ast [\alpha] = [\alpha]$ as well. Moreover, 
	\[
		([\alpha_1] \ast [\alpha_2]) \circ ([\beta_1] \ast [\beta_2]) = ([\alpha_1] \circ [\beta_1]) \ast ([\alpha_2] \circ [\beta_2])\ .
	\]
	Thus, the Eckmann-Hilton argument shows that $\ast$ and $\circ$ are the same operations on $\pi_0$ and they are both commutative (and associative). In particular, we have $[\alpha] \ast [\alpha^{-1}] = [\id{}]$. Note that 
	\[
		\kappa([\alpha_1] \ast [\alpha_2]) = [\psi^{-1}(\alpha_1(1 \otimes e) \otimes \alpha_2(1 \otimes e))] = [\alpha_1(1 \otimes e)] \cdot [\alpha_1(1 \otimes e)]
	\]
	by the definition of the ring structure on $K_0^\bT(D)$ and in particular 
	\[
		[1 \otimes e] = \kappa([\id{}]) = \kappa([\alpha] \ast [\alpha^{-1}]) = \kappa([\alpha]) \cdot \kappa([\alpha^{-1}])\ .
	\]
	Hence, $\kappa$ factors through a group homomorphism 
	\[
		\pi_0(\eqAut{\bT}{D \otimes \bK}) \to  GL_1(K_0^{\bT}(D)_+)
	\] 
	that we will continue to denote by $\kappa$. 
	
	The ring $K_0^\bT(D) \cong \Z[t,t^{-1},p_V(t)^{-1}]$ is a localisation of the unique factorisation domain $\Z[t]$ at the multiplicative subset $S$ generated by $\{t,p_V(t)\}$. Let $p_i(t) \in \Z[t]$ be the prime factors of the polynomial\footnote{As mentioned at the beginning of this section we may assume that $p_V(t) \in \Z[t]$.} $p_V(t) \in \Z[t]$. By Lem.~\ref{lem:units_in_loc} the group $GL_1(K_0^\bT(D))$ is therefore generated by $\pm 1$, $t$ and all $p_i(t) \in \Z[t]$. Consequently, $GL_1(K_0^\bT(D)_+)$ has the generators $\{t, p_1(t), \dots, p_\ell(t) \}$. To see that $\kappa$ is surjective it suffices to find automorphisms $\alpha^t, \alpha^{p_1}, \dots \alpha^{p_\ell} \in \eqAut{\bT}{D \otimes \bK}$ with the property that $x = [\alpha^x(1 \otimes e)] \in K_0^\bT(D)$. Let $W \colon H \to H$ be the unitary operator given by $W(\delta_k \otimes \eta) = \delta_{k+1} \otimes \eta$ and let $\alpha^t = \id{D} \otimes \Ad_W$. Then $\alpha^t$ is $\bT$-equivariant, and since $WeW^*$ is a projection onto a one-dimensional subspace of $H_1$, we have $[\alpha^t(1 \otimes e)] = t$. To define $\alpha^{p_i}$ note that the prime factor decomposition $\prod_{i=1}^\ell p_i(t)^{m_i}$ of $p_V(t)$ corresponds to a tensor product decomposition of $V$:
	\[
		V \cong \bigotimes_{i = 1}^\ell V_i^{\otimes m_i}\ .
	\]
	This induces an isomorphism for the endomorphism algebra of the form
	\[
		\Endo{V} \cong \bigotimes_{i = 1}^\ell \Endo{V_i}^{\otimes m_i}
	\]
	and shifting the tensor factors gives a $\bT$-equivariant unital $*$-isomorphism
	\[
		\psi_i \colon \Endo{V_i} \otimes D \to D\ . 
	\]
	Since $H$ contains each irreducible representation with infinite multiplicty there exists a unitary intertwiner $U_i \colon  H \otimes V_i  \to H$ that induces by conjugation the $\bT$-equivariant $*$-isomorphism 
	\[
		\phi_i \colon \bK \otimes \Endo{V_i} \to \bK 
	\]
	Define $\alpha^{p_i} \in \eqAut{\bT}{\bK \otimes D} \cong \eqAut{\bT}{D \otimes \bK}$ by $\alpha^{p_i} = (\phi_i \otimes \id{D}) \circ (\id{\bK} \otimes \psi_i^{-1})$ and note that the image of the projection $\alpha_{p_i}(e \otimes 1_D)$ is $\bT$-equivariantly isomorphic to $V_i \otimes D$. Therefore $[\alpha_{p_i}(e \otimes 1_D)] = p_i(t) \in K_0^\bT(D)$.
	
	Let $\alpha \in \eqAut{\bT}{D \otimes \bK}$ with $\kappa([\alpha]) = [1 \otimes e]$. Let $p = \alpha(1 \otimes e) \in (D \otimes \bK)^{\bT}$. The algebra $(D \otimes \bK)^{\bT}$ is an inductive limit of finite-dimensional algebras, which all have the cancellation property. Thus, this is also true for $(D \otimes \bK)^{\bT}$. By \cite[Prop.~3.4]{book:EvansKawahigashi} there exists a unitary $u \in U(M((D \otimes \bK)^{\bT}))$ with $upu^* = 1 \otimes e$. In fact $u$ can be chosen to lie in the unitary group of the unitization of $(D \otimes \bK)^\bT$. Since
	\[
		(D \otimes \bK)^{\bT} \cong (D \otimes \bK(\ell^2\Z))^\bT \otimes \bK(H_0)\ ,
	\]
	the group $U(M((D \otimes \bK)^{\bT}))$ is contractible in the strict topology. In particular there is a path from $1$ to $u$ in $U(M((D \otimes \bK)^{\bT}))$, which is continuous with respect to the strict topology.  Therefore there is also a point-norm continuous path in the $\bT$-equivariant automorphisms between $\alpha$ and $\beta = \Ad_u \circ \alpha$. But $\beta(1 \otimes e) = 1 \otimes e$, so by Thm.~\ref{thm:Aut_eq_stabiliser_contractible} there is a continuous path between $\id{D \otimes \bK}$ and $\beta$. This shows that $\kappa$ is injective.
\end{proof}

\section{The homotopy type of $\eqAutId{\bT}{D \otimes \bK}$} \label{sec:htpy-type}
In this section we will extend some of the classification results Dadarlat and the second author obtained in \cite{paper:DadarlatP-DD-theory} to the $\bT$-equivariant setting. This will allow us to determine the homotopy type of all path-components of $\eqAut{\bT}{D \otimes \bK}$. Let $\eqAutId{\bT}{D \otimes \bK}$ be the path-component of $\id{D \otimes \bK}$ and denote by $\Proj{p_0}{A}$ the path-component of $p_0$ in the space of self-adjoint projections in $A$ equipped with the norm topology.

\begin{prop} \label{prop:princ_bdl}
The evaluation map
\[
	q \colon \eqAutId{\bT}{D \otimes \bK} \to \Proj{1 \otimes e}{(D \otimes \bK)^\bT} \quad , \quad \alpha \mapsto \alpha(1 \otimes e)
\]
gives rise to a principal bundle over $\Proj{1 \otimes e}{(D \otimes \bK)^\bT}$ with structure group $\eqAut{\bT,1 \otimes e}{D \otimes \bK}$. 
\end{prop}

\begin{proof}
	The key issue here is local triviality, which is proven in the same way as \cite[Lem.~2.8]{paper:DadarlatP-DD-theory}. Indeed, $q$ has local sections constructed as follows:	Let  
	\[
		W = \left\{ p \in \Proj{1 \otimes e}{(D \otimes \bK)^\bT}\ |\ \lVert p - p_0 \rVert < 1\right\} \ .
	\] 
	The unitary element $u_p = x_p(x_p^*x_p)^{-\frac{1}{2}}$ with $x_p = pp_0 + (1-p)(1-p_0)$ in $M(D \otimes \bK)$ associated to any $p \in W$ is invariant under the group action, satisfies $\Ad_{u_p}(p_0) = u_pp_0u_p^* = p$ and $p \mapsto \Ad_{u_p}$ is continuous.
\end{proof}

In the following proposition we will consider $U(M( (D \otimes \bK)^\bT ))$ equipped with the norm topology. As a space this topological group is contractible by the main theorem of \cite{chapter:CuntzHigson}. Moreover, $U(D^\bT)$ embeds into it via
\[
	U(D^\bT) \to U(M( (D \otimes \bK)^\bT )) \quad , \quad v \mapsto v \otimes e + (1 - 1\otimes e)\ .
\]

\begin{prop} \label{prop:Proj_BU}
Let $A = (D\otimes \bK)^\bT$ and let 
\[
	U_{1 \otimes e}(M(A)) = \{ u \in U(M(A)) \ |\ u(1 \otimes e) = (1 \otimes e)u \}
\] 
(with the norm topology). The map
\[
	\Ad{} \colon U(M(A)) \to \Proj{1 \otimes e}{A} \quad , \quad u \mapsto u(1 \otimes e)u^*
\]
gives rise to a principal bundle with structure group $U_{1 \otimes e}(M(A))$. Moreover, the inclusion $U(D^\bT) \to U_{1 \otimes e}(M(A))$ is a homotopy equivalence and therefore 
\[
	\Proj{1 \otimes e}{A} \simeq BU(D^\bT)\ .
\]
\end{prop}

\begin{proof}
	The argument for local triviality is the same as in the proof of Prop.~\ref{prop:princ_bdl} and therefore traces back to \cite[Lem.~2.8]{paper:DadarlatP-DD-theory} as well: Over the set
	\[
		W = \left\{ p \in \Proj{1 \otimes e}{A}\ |\ \lVert p - p_0 \rVert < 1\right\} \ ,
	\] 
	we can define a norm-continuous map to $U(M(A))$ given by $p \mapsto u_p$, with $u_p$ as in the proof of Prop.~\ref{prop:princ_bdl}, that provides a local section of $\Ad{}$. It is also clear from the definition that $U_{1 \otimes e}(M(A))$ is the fibre of $\Ad{}$ over $1 \otimes e$.
	
	It remains to be seen why $U(D^\bT) \to U_{1 \otimes e}(M(A))$ is a homotopy equivalence. Let $p_0 = 1 \otimes e$ and note that 
	\[
		U_{p_0}(M(A)) \cong U(p_0 M(A) p_0) \times U((1 - p_0) M(A) (1 - p_0))\ .
	\]
	However, since $p_0 \in A$ we have $p_0 M(A) p_0 = p_0 A p_0 = (p_0(D \otimes \bK)p_0)^\bT = D^\bT$. Let $p_0^\perp = 1 - p_0 = 1_D \otimes (1-e) = 1_D \otimes e^\perp$. By \cite[Cor.~1.2.37]{book:AraMathieu} the corner of $M(A)$ corresponding to $p_0^\perp$ is
	\[
		p_0^\perp M(A) p_0^\perp = M(p_0^\perp A p_0^\perp)\ .
	\]
	Since $e \in \bK$ projects down to a trivial one-dimensional subrepresentation of~$H$, and $H$ contains each representation with infinite multiplicity, there is a $\bT$-equivariant isomorphism $\bK(e^\perp H) \cong \bK(H)$. With $e^\perp \bK(H) e^\perp = \bK(e^\perp H)$ we obtain
	\[  
	 	p_0^\perp A p_0^\perp = (D \otimes e^\perp \bK e^\perp)^\bT \cong A\ ,
	\] 
	which implies that $U(p_0^\perp M(A) p_0^\perp) \cong U(M(A))$ is contractible. Hence, the inclusion $U(D^\bT) \to U_{p_0}(M(A))$ is a homotopy equivalence with homotopy inverse given by the projection $U_{p_0}(M(A)) \to U(p_0 M(A) p_0) \cong U(D^\bT)$.
	
	To see why the final statement in the proposition holds note that the contractibility of $U(M(A))$ implies that $\Proj{p_0}{A} \simeq BU_{p_0}(M(A))$. But the map $BU(D^\bT) \to BU_{p_0}(M(A))$ induced by the inclusion on classifying spaces is a homotopy equivalence by our previous observations.
\end{proof}

\begin{prop} \label{prop:Aut_Proj_equivalence}
The map $q \colon \eqAutId{\bT}{D \otimes \bK} \to \Proj{1 \otimes e}{(D \otimes \bK)^\bT}$ defined by $q(\alpha) = \alpha(1 \otimes e)$ as in Prop.~\ref{prop:princ_bdl} induces a homotopy equivalence $\eqAutId{\bT}{D \otimes \bK} \simeq \Proj{1 \otimes e}{(D \otimes \bK)^\bT}$.
\end{prop}

\begin{proof}
	This is proven in the same way as \cite[Cor.~2.9]{paper:DadarlatP-DD-theory}. We give a summary here: The space $\Proj{1 \otimes e}{(D \otimes \bK)^\bT}$ is a Banach manifold and therefore has the homotopy type of a CW-complex. This implies that the fibre and the base space of the principal bundle with bundle projection $q$ both have the homotopy type of CW-complexes. By \cite[Thm.~2]{paper:Schoen} this then also holds for the space $\eqAutId{\bT}{D \otimes \bK}$ and the long exact sequence of homotopy groups together with Whitehead's theorem proves the result.
\end{proof}

We point out the following technical lemma that will be important in the final section:
\begin{lemma} \label{lem:eqAut_well-pointed}
	The topological group $\eqAutId{\bT}{D \otimes \bK}$ is well-pointed in the sense that the inclusion of the identity into the group is a cofibration, ie.\ the pair $(\eqAutId{\bT}{D \otimes \bK},\id{D \otimes \bK})$ is a neighbourhood deformation retract.
\end{lemma}
\begin{proof}
	The corresponding non-equivariant statement is \cite[Prop.~2.26]{paper:DadarlatP-DD-theory}. The proof works verbatim in the equivariant case, since it is based on the following observations: The map $q \colon \eqAutId{\bT}{D \otimes \bK} \to \Proj{1 \otimes e}{(D \otimes \bK)^\bT}$ is a principal bundle with structure group $\eqAut{\bT,1 \otimes e}{D \otimes \bK}$ (by Prop.~\ref{prop:princ_bdl}) that is trivialisable over $W = \{ p \in \Proj{p_0}{(D \otimes \bK)^\bT} \ |\ \lVert p - p_0 \rVert < 1/2 \}$ with $p_0 = 1 \otimes e$. The group $\eqAut{\bT,1 \otimes e}{D \otimes \bK}$ deformation retracts onto its identity element via the contraction described in Thm.~\ref{thm:Aut_eq_stabiliser_contractible}.
\end{proof}

As explained in a paragraph before Lem.~\ref{lem:pi0_equiv} the $\bT$-equivariant $K$-groups of $D$ are $K_0^\bT(D) \cong \Z[t,t^{-1},p_V(t)^{-1}]$ and $K_1^\bT(D) = 0$. The isomorphism with $K_0^\bT(D)$ induces the following order structure on $\Z[t,t^{-1},p_V(t)^{-1}]$: 
\[
	\frac{q(t)}{t^\ell p_V(t)^k} \geq 0 \  \Leftrightarrow \  q(t)p_V(t)^m \text{ has non-negative coefficients for some } m\in \N \ .
\] 
The last lemma suggests that we will also need to compute the $K$-groups of the fixed-point algebra $D^\bT$. For the case $p_V(t) = 1 + t$ this amounts to computing the $K_0$-group of the GICAR algebra (see \cite[Sec.~5]{paper:Bratteli} for the definition, \cite[p.~149]{book:Renault} for its $K_0$-group and \cite[Ex.~IV.3.7]{book:Davidson} for how it is related to our picture of it). Motivated by this result we define the bounded subring of $K_0^\bT(D)$ to be 
\[
	\bddR = \left\{ x \in K_0^\bT(D) \ | \ -m[1_D] \leq x \leq m[1_D] \text{ for some } m \in \N \right\}\ .
\]

\begin{lemma} \label{lem:K_groups_fixed_point_algebra}
Let $D = (\Endo{V})^{\otimes \infty}$ be the infinite UHF-$\bT$-algebra associated to a $\bT$-representation $V$ corresponding to $p_V(t) \in \Rep{\bT} \cong \Z[t,t^{-1}]$. Then the $K$-groups of the fixed-point algebra are given by
\[
	K_0(D^\bT) \cong \bddR \qquad \text{and} \qquad K_1(D^\bT) = 0\ .
\]
\end{lemma}

\begin{proof}
	Since $\bT$ acts on each tensor factor $\Endo{V}$ separately, the fixed-point algebra $D^\bT$ is isomorphic to the direct limit
	\[
		\Endo{V}^\bT \to \dots \to (\Endo{V}^{\otimes n})^\bT \to (\Endo{V}^{\otimes (n+1)})^\bT \to \dots\ ,
	\]
	where the connecting $*$-homomorphisms are given by $T \mapsto T \otimes 1$. Note that $(\Endo{V}^{\otimes n})^\bT \cong \hom_{\bT}(V^{\otimes n}, V^{\otimes n})$, which is isomorphic to a direct sum of matrix algebras. By Schur's lemma this turns out to be
	\[
		 \bigoplus_{k \in \Z} \hom_{\bT}(V_{k}, V_{k}) \cong \bigoplus_{i \in I} M_{n_i}(\C) \otimes \hom_{\bT}(\C_{m_i}, \C_{m_i}) \cong \bigoplus_{i \in I} M_{n_i}(\C)
	\]
	where $V_{m_i} \cong \C^{n_i} \otimes \C_{m_i}$ (compare with the decomposition in \eqref{eqn:decomposition}), so $n_i$ is the multiplicity of the $\C_{m_i}$ in $V$. This implies that $K_0((\Endo{V}^{\otimes n})^\bT)$ is the free abelian group generated by the irreducible subrepresentations of $V^{\otimes n}$. We claim that 
	\begin{align} \label{eqn:bddRn}
		 & K_0((\Endo{V}^{\otimes n})^\bT) \\
		\cong & \left\{ q \in \Z[t,t^{-1}] \ | \ -m\,p_V^n \leq q \leq m\,p_V^n \text{ for some } m \in \N \right\} =: \bddR^{(n)} \notag
	\end{align}
	To see why this is true note first that the polynomials $t^{m_i}$ corresponding to the irreducible subrepresentations of $V^{\otimes n}$ satisfy $-p_V^n \leq t^{m_i} \leq p_V^n$. Therefore $K_0((\Endo{V}^{\otimes n})^\bT)$ is a subgroup of $R_{\text{bdd}}^{(n)}$. Now 	let $q \in \Z[t,t^{-1}]$ be a polynomial so that there exists $m \in \N$ with $-m\,p_V^n \leq q \leq m\,p_V^n$. We may write $q(t) = q_1(t) - q_2(t)$ where $q_i(t)$ are polynomials with non-negative coefficients, which we will denote by $a_k$ for the coefficient of $t^k$ in $q_1$ and $b_k$ for $q_2$. These polynomials may be chosen such that $a_k \neq 0 \Rightarrow b_k = 0$ and $b_k \neq 0 \Rightarrow a_k = 0$. The inequality
	\[
		q_1 - q_2 \leq m\,p_V^n \quad \Leftrightarrow (m\,p_V^n - q_1) + q_2 \geq 0
	\]
	can then only hold if the term in brackets is non-negative, which implies that the representation corresponding to $q_1$ is a subrepresentation of $\bigoplus_{k=1}^m V^{\otimes n}$. Repeating this argument with the other inequality shows that the same must be true for $q_2$. Hence, the polynomials $q_i(t)$ correspond to representations that are elements of $K_0((\Endo{V}^{\otimes n})^\bT)$ proving the claim.

	Identifying $K_0((\Endo{V}^{\otimes n})^\bT)$ with $\bddR^{(n)}$ the connecting homomorphism $\bddR^{(n)} \to \bddR^{(n+1)}$ in the direct limit becomes multiplication by $p_V(t)$. The maps
	\(
		\varphi_n \colon \bddR^{(n)} \to \bddR \ , \  q \mapsto \frac{q}{p_V^n}
	\)
	fit into the following commutative diagram
	\[
	\begin{tikzcd}
			\bddR^{(n)} \ar[dr, "\varphi_n" left] \ar[rr, "q \mapsto q \cdot p_V"] & & \bddR^{(n+1)} \ar[dl, "\varphi_{n+1}"] \\
			& \bddR
	\end{tikzcd}
	\]
	which defines a homomorphism $\varphi \colon K_0(D^\bT) \to \bddR$. To construct an inverse identify $K_0^\bT(D)$ with $\Z[t,t^{-1},p_V(t)^{-1}]$ and 	consider $r \in \bddR$ such that $-m \leq r \leq m$. Then $r$ is of the form 
	\[
		r(t) = \frac{q(t)}{t^\ell p_V(t)^{k}}
	\]
	with $k,\ell \geq 0$ and $q \in \Z[t]$, and we must have 
	\[
		-m\,t^{\ell}\,p_V(t)^{k+m} \leq q(t)p_V(t)^m \leq m\,t^{\ell}\,p_V(t)^{k+m}\ .
	\]
	Let $q(t)p_V(t)^m = \sum_{i \in \N_0} a_i t^i$ and $p_V(t)^k = \sum_{j \in \N_0} b_j t^j$. The above inequality implies $b_i = 0 \Rightarrow a_{i+\ell} =0$. In particular, $a_0 = \dots = a_{\ell-1} = 0$, ie.\ $q(t)p_V(t)^m$ is divisible by~$t^\ell$ and  
	\[
		\frac{q(t)p_V(t)^m}{t^\ell} \in \bddR^{(k+m)}\ .
	\]
	The resulting element in the colimit is independent of the choice of the fraction $\frac{q(t)}{t^\ell p_V(t)^{k}}$ representing $r(t)$. This provides the inverse homomorphism $\bddR \to K_0(D^\bT)$ of $\varphi$, which finishes the computation of $K_0(D^\bT)$, but since $D^\bT$ is an AF-algebra we also have $K_1(D^\bT) = 0$.
\end{proof}

\begin{corollary} \label{cor:bddR_in_Zt}
	The bounded subring $\bddR \subset K_0^\bT(D) \cong \Z[t,t^{-1},p_V(t)^{-1}]$ satifies 
	\[
		\bddR \subset \Z[t,p_V(t)^{-1}]\ .
	\]
\end{corollary}
\begin{proof}
	This is a consequence of the observation in the proof of Lemma~\ref{lem:K_groups_fixed_point_algebra} that any $r \in \bddR \subset \Z[t,t^{-1},p_V(t)^{-1}]$ given by a quotient  
	\[
		r(t) = \frac{q(t)}{t^\ell p_V(t)^{k}}
	\]
	with $k,\ell \geq 0$ and $q \in \Z[t]$ must have $q(t)$ divisible by $t^\ell$.
\end{proof}

Prop.~\ref{prop:princ_bdl} and Prop.~\ref{prop:Proj_BU} reduce the computation of the homotopy groups of $\eqAutId{\bT}{D \otimes \bK}$ to the ones of $U(D^\bT)$. This is essentially a computation of non-stable $K$-theory in the sense of \cite{paper:ThomsenNonstable}. Interestingly, the higher homotopy groups will in general only be $2$-periodic from degree $3$ onwards. More precisely, we always have $\pi_1(U(D^\bT)) \cong K_0(D^\bT)$, but, depending on the polynomial $p_V(t)$, the groups $\pi_{2k-1}(U(D^\bT))$ for $k > 1$ will be all isomorphic to the same subgroup of $K_0(D^\bT)$. Let
\begin{align*}
	\bddR^0 &= \{r \in \bddR \ |\ r(0) = 0 \} \ ,\\
	\bddR^\infty &= \left\{\frac{q}{p_V^k} \in \bddR \ |\ q \in \Z[t],\ k \geq 0,\ \deg(q) < kd \right\}\ .
\end{align*}

\begin{theorem} \label{thm:coefficients}
	Let $D = \Endo{V}^{\otimes \infty}$ be the infinite UHF-algebra associated to a $\bT$-representation $V$ and equipped with its natural $\bT$-action. Let 
	\[
		p_V(t) = \sum_{i=0}^d a_i t^i \in \Z[t] \subset  \Z[t,t^{-1}] \cong \Rep{\bT}
	\]
	be the associated polynomial of degree $d$. Then $\pi_{2k}(U(D^\bT)) = 0$ for $k \in \N_0$ and
	\begin{equation} \label{eqn:pi1}
		\pi_1(U(D^\bT)) \cong K_0(D^\bT) \cong \bddR \ .
	\end{equation}
	The odd homotopy groups $\pi_{2k-1}(U(D^\bT))$ for $k > 1$ depend on $a_0$ and $a_d$ in the following way:
	\begin{enumerate}[a)]
		\item If $a_0 > 1$ and $a_d > 1$, then $\pi_{2k-1}(U(D^\bT)) \cong \bddR$.
		\item If $a_0 = 1$ and $a_d > 1$, then $\pi_{2k-1}(U(D^\bT)) \cong \bddR^0$.
		\item If $a_0 > 1$ and $a_d = 1$, then $\pi_{2k-1}(U(D^\bT)) \cong \bddR^\infty$. 
		\item If $a_0 = 1$ and $a_d = 1$, then $\pi_{2k-1}(U(D^\bT)) \cong \bddR^\infty \cap \bddR^0$.
	\end{enumerate}
\end{theorem} 

\begin{proof}
	The isomorphism \eqref{eqn:pi1} was shown in \cite[Thm.~4.6]{paper:HandelmanK0AF}. For the computation of the higher homotopy groups we note that by \cite[Prop.~4.4]{paper:HandelmanK0AF} the group $\pi_m(U(D^\bT))$ is the direct limit of the sequence
	\[
		\pi_m(U(A_{(k)}^\bT)) \to \pi_m(U(A_{(k+1)}^\bT))
	\]
	where $A_{(k)} = \Endo{V}^{\otimes k} \cong \Endo{V^{\otimes k}}$ and $A_{(k)}^\bT$ is the fixed-point algebra. The coefficients $a_i$ of $p_V$ can be interpreted as multiplicities as follows: We have $a_i = \dim(V_i)$ with $V_i$ as in the decomposition \eqref{eqn:decomposition}. Thus, if 
	\[
		p_V(t)^k = \sum_{\ell = 0}^{kd} b_{\ell} t^\ell\ ,
	\]
	then by Schur's lemma 
	\[
		A_{(k)}^\bT \cong \bigoplus_{\ell = 0}^{kd} \Endo{W_{\ell}}
	\]
	where $W_{\ell} \subset V^{\otimes k}$ is the subrepresentation corresponding to the character $z \mapsto z^\ell$ and satisfies $\dim(W_{\ell}) = b_{\ell}$. Hence,
	\[
		U(A_{(k)}^\bT) \cong \prod_{\ell=0}^{kd} U(b_{\ell})\ .
	\]
	The polynomial $p_V(t)$ has non-negative integer coefficients. By Lem.~\ref{lem:large_middle} there is an $N \in \N$ such that $p_V(t)^N$ has large middle coefficients in the sense of Def.~\ref{def:large_middle}. Since the sequence
	\[
		A_{(0)}^\bT \to A_{(N)}^\bT \to \dots \to A_{(mN)}^\bT \to A_{((m+1)N)}^\bT \to \dots 
	\]
	with connecting homomorphisms given by $T \mapsto T \otimes 1$ also has direct limit $D^\bT$ we may work with $V^{\otimes N}$ instead of $V$ and assume without loss of generality that $p_V$ has large middle coefficients. 
	
	By Lem.~\ref{lem:coefficients_grow} the coefficients $b_\ell$ of $p_V(t)^k$ then satisfy $b_\ell = 0$ or $b_\ell > k$ for $\ell \in \{1, \dots, kd-1\}$ and $k \in \N$. If $a_0 > 1$, then we also have $b_0 > k$ for the lowest order coefficient. If $a_d > 1$, then $b_{kd} > k$ for the highest order coefficient as well. This implies that for $k > \frac{m}{2}$ the unitary groups $U(b_{\ell})$ will all have dimension $b_{\ell} > \frac{m}{2}$ except potentially the ones corresponding to the lowest and the highest order coefficients of $p_V$, which might be $1$. Let $U_\infty$ be the colimit over the inclusions $U(n) \to U(n+1)$ that add a $1$ in the lower right hand corner. In the above situation $m$ falls into the stable range for the homotopy groups, which implies that $\pi_m(U(b_{\ell})) \cong \pi_m(U_\infty)$ for $\ell \in \{1, \dots, kd-1\}$, which is isomorphic to $\Z$ if $m$ is odd and vanishes if $m$ is even.
	
	Let $\bddR^{(k)} \cong K_0(A_{(k)}^\bT)$ be as in \eqref{eqn:bddRn}. If $m$ is odd and $k > \frac{m}{2}$, then the first homomorphism in the following composition
	\[
	\begin{tikzcd}
		\pi_m(U(A_{(k)}^\bT)) \ar[r] & K_1(S^m A_{(k)}^\bT) \ar[r,"\cong"] & \widetilde{K}_0(S^{m+1}A_{(k)}^\bT) \ar[r,"\cong"] & \bddR^{(k)}
	\end{tikzcd}
	\]
	is injective. The last isomorphism in this chain uses Bott periodicity and the identification with $\bddR^{(k)}$. The image of this map depends on the lowest and highest order coefficients in $p_V$. If $a_0 = 1$, then $\pi_m(U(b_0))=\pi_m(U(1))$ vanishes for $m > 1$ and therefore the image will contain only polynomials with vanishing constant term. Likewise, if $a_d = 1$, then the polynomials in the image must have vanishing highest order term, or in other words their degree must be strictly lower than $kd$. If $a_0 > 1$ and $a_d > 1$, then the above map is an isomorphism. Since the homomorphism $\bddR^{(k)} \to \bddR$ is given by division by $p_V^k$ the result follows after taking the direct limit. Note that for the direct limit we can neglect all terms in the sequence where $k \leq \frac{m}{2}$.
\end{proof}

\begin{remark} \label{rem:Z-absorbtion}
	The case distinction in Thm.~\ref{thm:coefficients} is closely linked to the question whether $D^\bT$ tensorially absorbs the Jiang-Su algebra $\JiangSu$. The $k$th~level of the Bratteli diagram for the AF-algebra $D^\bT$ has one vertex for each non-zero coefficient in $p_V(t)^k$. Using the same notation for the coefficients as in the proof of the previous theorem we have an edge of multiplicity $a_i$ in the diagram between the vertex corresponding to $b_i t^i$ in $p_V(t)^k$ and the one for $b_i a_j t^{i+j}$ in $p_V(t)^k \cdot p_V(t)$. If $a_0 = 1$, then the subset of all vertices corresponding to the coefficients $b_it^i$ with $i \neq 0$ provides a saturated hereditary subset that corresponds to a character $\chi \colon D^\bT \to \C$ by the classification of ideals in AF-algebras given in \cite[Thm.~3.3]{paper:Bratteli}. A similar argument works in the case where $a_d = 1$ using the complement of the vertices corresponding to highest order coefficients. Thus, if $a_0 = 1$ or $a_d = 1$, then $D^\bT \ncong D^\bT \otimes \JiangSu$. The special case $p_V(t) = 1+t$ (where $D^\bT$ is the GICAR algebra) has also been discussed in \cite[Ex.~6.4]{paper:GaborII}.

	In case $a_0 > 1$ and $a_d > 1$ the fixed-point algebra $B := \Endo{V^{\otimes k}}^\bT$ for sufficiently large $k>0$ is a completely non-commutative finite-dimensional $C^*$-algebra in the sense of \cite[Def.~1.1]{paper:BlackadarKumjianRordam} by Lem.~\ref{lem:large_middle} and \ref{lem:coefficients_grow} in the appendix and we have $D \cong \Endo{V^{\otimes k}}^{\otimes \infty}$. Each of the $*$-homomorphisms 
	\[
		\varphi_i \colon B \to D \qquad , \qquad a \mapsto 1 \otimes \dots \otimes 1 \otimes a \otimes 1 \otimes \dots
	\] 
	factors through the fixed-point algebra $D^\bT$ and embeds $B$ into $D^\bT$ in such a way that the image eventually commutes up to $\epsilon > 0$ with any finite set of elements in $D^\bT$. This shows that $D^\bT$ is approximately divisible in the sense of \cite[Def.~1.2]{paper:BlackadarKumjianRordam}. Hence, $D^\bT \cong D^\bT \otimes \JiangSu$ by \cite[Thm.~2.3]{paper:TomsWinterASH} if $a_0 > 1$ and $a_d > 1$. It would be interesting to see whether there exists a $\bT$-equivariant $*$-isomorphism $D \cong D \otimes \JiangSu$ in this case (with $\bT$ acting trivially on $\JiangSu$). 
	
\end{remark}

\section{Classification of $\bT$-equivariant $D \otimes \bK$-bundles} \label{sec:classification}
In this final section we will classify isomorphism classes of locally trivial bundles of $C^*$-algebras with fibre $D\otimes \bK$ that carry the $\bT$-action in each fibre discussed in the previous sections (ie.\ they allow local trivialisations such that the transition functions are equivariant with respect to that action). The invariant associated to such a bundle will be an element in a generalised cohomology theory. It arises from an infinite loop space structure on the classifying space $B\!\eqAut{\bT}{D \otimes \bK}$, and its coefficients are given by the homotopy groups determined in the previous section. We start by recalling a few basic facts from stable homotopy theory \cite{book:Adams-InfLoopSp}.
\begin{definition} \label{def:inf_loop_space}
	A sequence of pointed topological spaces $(Y_k)_{k \in \N_0}$ together with weak homotopy equivalences 
	\[
		\sigma_k \colon Y_k \to \Omega Y_{k+1} 
	\] 
	(where $\Omega Z$ denotes the based loop space of the pointed space $Z$) is called an \emph{$\Omega$-spectrum}. If $Y_0$ is the zeroth space of an $\Omega$-spectrum it is called an \emph{infinite loop space}.
\end{definition}

Each $\Omega$-spectrum $(Y_k)_{k \in \N_0}$ gives rise to a reduced cohomology theory on the category of pointed finite CW-complexes in the following way
\[
	\widetilde{h}^k(X) = [X, Y_k]_+
\]
where the right hand side denotes based homotopy classes of base point preserving continuous maps and we define $Y_{n} = \Omega^{-n} Y_0$ for $n < 0$. Note that this indeed gives abelian groups, since $\widetilde{h}^k(X) \cong [X, \Omega^2 Y_{k+2}]_+$ and homotopy classes of maps into double loop spaces carry a natural abelian group structure. The coefficients $\check{h}^*$ of the theory are defined to be $\check{h}^k = \widetilde{h}^k(S^0)$. Using the loop-suspension adjunction we see that $\check{h}^k = \pi_{-k}(Y_0)$, so the coefficients are fixed by the homotopy groups of the infinite loop space underlying $\widetilde{h}^*$.

Note that $X \mapsto \widetilde{h}^*(X)$ also gives rise to an unreduced cohomology theory on pairs of spaces $(X,A)$: Let $CA = A \times I/A \times \{0\}$ be the cone on $A$, denote by $X_+$ the pointed space obtained from $X$ by adding a disjoint base point and define 
\[
	h^k(X,A) = \widetilde{h}^k(X_+ \cup C(A_+))
\]  
where the point $[a,1] \in C(A_+)$ is identified with $a \in X_+$  in $X_+ \cup C(A_+)$. By definition
\[
	h^k(X) = h^k(X,\emptyset) \cong \widetilde{h}^k(X_+) = [X_+, Y_k]_+ \cong [X, Y_k]\ ,
\] 
where the right hand side denotes the unbased homotopy classes. Moreover, if $A \subset X$ is the inclusion of a subcomplex, then $h^k(X,A) \cong \widetilde{h}^k(X/A)$.

There are several ``infinite loop space machines'' in algebraic topology that produce $\Omega$-spectra and therefore cohomology theories. The one that is most suited to our setting relies on diagram spaces for a diagram category of finite sets and injections. More precisely, let $\cI$ be the category with objects given by $\mathbf{n} = \{1,\dots,n\}$ for $n \in \N_0$ (the objects include $\mathbf{0} = \emptyset$) and morphisms given by injective maps between the sets. An $\cI$-space is a (covariant) functor from $\cI$ to the category of topological spaces. 

There is a symmetric monoidal structure, denoted by $\sqcup$, on $\cI$ defined as follows: For $\mathbf{m}, \mathbf{n} \in \obj{\cI}$ let $\mathbf{m} \sqcup \mathbf{n} = \{1,\dots,m+n\}$ on objects. Let $f \in \hom(\mathbf{m}, \mathbf{m}'), g \in \hom(\mathbf{n},\mathbf{n}')$. The  morphism $f \sqcup g \colon \mathbf{m} \sqcup \mathbf{n} \to \mathbf{m}' \sqcup \mathbf{n}'$ acts by identifying $\mathbf{m}$ with the first $m$ elements of $\{1, \dots, m+n\}$ (and similarly for $\mathbf{m}'$) and $\mathbf{n}$ with the last $n$ elements (similarly for $\mathbf{n}'$). The symmetry $\tau_{m,n}\colon \mathbf{m} \sqcup \mathbf{n} \to \mathbf{n} \sqcup \mathbf{m}$ is defined by an $(m,n)$-block permutation. 

The symmetric monoidal structure on $\cI$ together with the cartesian product on topological spaces turns the category of $\cI$-spaces into a symmetric monoidal category. A monoid in this category is called an $\cI$-monoid. Unravelling the definitions, an $\cI$-monoid is an $\cI$-space $X$ together with a natural transformation
\[
	\mu_{m,n} \colon X(\mathbf{m}) \times X(\mathbf{n}) \to X(\mathbf{m} \sqcup \mathbf{n})\ . 
\]
that makes the obvious associativity and unitality diagrams commute. The $\cI$-monoid is called commutative if the following diagram commutes
\[
\begin{tikzcd}
	X(\mathbf{m}) \times X(\mathbf{n}) \ar[r,"\mu_{m,n}"] \ar[d] & X(\mathbf{m} \sqcup \mathbf{n}) \ar[d,"X(\tau_{m,n})"] \\
	X(\mathbf{n}) \times X(\mathbf{m}) \ar[r,"\mu_{n,m}" below]& X(\mathbf{n} \sqcup \mathbf{m}) 
\end{tikzcd}
\]
where the vertical arrow on the left is interchanging the two factors. As described in detail in \cite[Sec.~5.2]{paper:Schlichtkrull} any commutative $\cI$-monoid $X$ gives rise to a $\Gamma$-space $\Gamma(X)$ in the sense of \cite{paper:SegalCatCoh}. A $\Gamma$-space is a diagram in topological spaces indexed by the category $\Gamma^{\rm op}$ of pointed finite sets (more precisely, a skeleton thereof). It can be seen as a homotopy theoretic generalisation of a commutative monoid. The information contained in the diagram allows to define infinite deloopings. Let $X$ be a commutative $\cI$-monoid that is grouplike in the sense that the monoid obtained as the homotopy colimit
\[
	X_{h\cI} = \hocolim_{\cI} X
\]
is a group (see \cite{paper:Vogt-hocolim} for a reference). It was shown in \cite[Prop.~5.3]{paper:Schlichtkrull} that in this case $X_{h\cI}$ is an infinite loop space. More precisely, it is the zeroth space of the $\Omega$-spectrum obtained from the $\Gamma$-space associated to $X$. In short, any grouplike commutative $\cI$-monoid $X$ gives rise to a cohomology theory $h^*$ with $\check{h}^k \cong \pi_{-k}(X_{h\cI})$.

\subsection{The commutative $\cI$-monoid $G_D^\bT$}
Let $V$ be a finite-dimensional $\bT$-representation as in \eqref{eqn:decomposition}. Let $D = \Endo{V}^{\otimes \infty}$ be the associated infinite UHF-algebra with its canonical $\bT$-action. Let $H = \ell^2(\Z) \otimes H_0$ be the $\bT$-representation that contains each irreducible representation with infinite multiplicity and define $\bK = \bK(H)$ to be the compact operators on $H$. Let 
\begin{equation} \label{eqn:GD-I-monoid}
	G_D^\bT(\mathbf{n}) = \eqAut{\bT}{(D \otimes \bK)^{\otimes n}}
\end{equation}
Any morphism $f \colon \mathbf{m} \to \mathbf{n}$ in $\cI$ can be decomposed as $f = \sigma \circ \iota$, where $\iota \colon \mathbf{m} \to \mathbf{n}$ is the inclusion that identifies $\{1, \dots, m\}$ with the first $m$ elements of $\mathbf{n}$ (note that we have $m \leq n$ otherwise the morphism set is empty) and $\sigma \colon \mathbf{n} \to \mathbf{n}$ is a permutation. Define $\hat{\sigma} \colon (D \otimes \bK)^{\otimes n} \to (D \otimes \bK)^{\otimes n}$ to be the $\bT$-equivariant $*$-automorphism that permutes the tensor factors according to $\sigma$ and let $G_D^\bT(\sigma) \colon G_D^\bT(\mathbf{n}) \to G_D^\bT(\mathbf{n})$ be given by conjugation by $\hat{\sigma}$. Let
\[
	G_D^\bT(\iota) \colon G_D^\bT(\mathbf{m}) \to G_D^\bT(\mathbf{n}) \quad, \quad \alpha \mapsto \alpha \otimes \id{(D \otimes \bK)^{n-m}}\ .
\]
A straightforward computation shows that $G_D^\bT(\sigma) \circ G_D^\bT(\iota) = G_D^\bT(\sigma') \circ G_D^\bT(\iota)$ if $\sigma \circ \iota = \sigma' \circ \iota$. Therefore $G_D^\bT(f) = G_D^\bT(\sigma) \circ G_D^\bT(\iota)$ is well-defined and this definition extends $G_D^\bT$ to a functor $\cI \to \text{Grp}$, where $\text{Grp}$ is the category of topological groups. The multiplication
\[
	\mu_{m,n} \colon G_D^\bT(\mathbf{m}) \times G_D^\bT(\mathbf{n}) \to G_D^\bT(\mathbf{m} \sqcup \mathbf{n})
\]
defined by $\mu_{m,n}(\alpha, \beta) = \alpha \otimes \beta$ is associative and turns $G_D^\bT \colon \cI \to \text{Grp}$ into an $\cI$-monoid taking values in topological groups. Let $\tau_{m,n}$ be the symmetry of $\cI$ and note that $G_D^\bT(\tau_{m,n})(\alpha \otimes \beta) = \beta \otimes \alpha$. Hence, $G_D^\bT$ is a commutative $\cI$-monoid. In a similar way we define
\[
	G_{D,0}^\bT(\mathbf{n}) = \eqAutId{\bT}{(D \otimes \bK)^{\otimes n}}\ .
\]
The same arguments as above also apply to $G_{D,0}^\bT$ proving that it also is a commutative $\cI$-monoid. In fact, both of these diagram spaces have more structure that will prove to be crucial in Cor.~\ref{cor:coh_theory}.

\begin{lemma} \label{lem:stable_EHI-group}
	The commutative $\cI$-monoids $G_D^\bT$ and $G_{D,0}^\bT$ are both stable EH-$\cI$-groups with compatible inverses in the sense of \cite[Def.~3.1]{paper:DadarlatP-UnitSpectra}.
\end{lemma}

\begin{proof}
	We will only prove the statement for $G_D^\bT$ here, since the only difference to $G_{D,0}^\bT$ is that $\pi_0(G_{D,0}^\bT(\mathbf{n}))$ is trivial, which makes the proof easier.  
	
	Note that $\eqAut{\bT}{D \otimes \bK}$ is a well-pointed topological group by Lem.~\ref{lem:eqAut_well-pointed}. The identity
	\(
		(\alpha_1 \otimes \alpha_2) \circ (\beta_1 \otimes \beta_2) = (\alpha_1 \circ \beta_1) \otimes (\alpha_2 \circ \beta_2)
	\) 
	shows that the diagram in \cite[Def.~3.1]{paper:DadarlatP-UnitSpectra} commutes. Let $p_k = (1 \otimes e)^{\otimes k} \in (D \otimes \bK)^{\otimes k}$. To see that $G_D^\bT$ has compatible inverses note that 
	\[
		\pi_0(G^\bT_D(\mathbf{m} \sqcup \mathbf{m})) \cong \pi_0(\eqAut{\bT}{(D \otimes \bK)^{\otimes m}}) \cong GL_1(K_0^\bT(D^{\otimes m})_+)  
	\]
	similar to Lem.~\ref{lem:pi0_equiv}. The second isomorphism is given by $[\alpha] \mapsto [\alpha(p_m)]$ and the first isomorphism is defined by conjugation with a $\bT$-equivariant isomorphism 
	\[
		\psi \colon (D \otimes \bK)^{\otimes 2m} \to (D \otimes \bK)^{\otimes m}\ ,
	\]
	which may be chosen in such a way that $\psi( p_{2m} ) = p_m$. The $K$-theory classes corresponding to $(\iota_m \sqcup \id{\mathbf{m}})_*(\alpha) = \id{} \otimes \alpha$ and $(\id{\mathbf{m}} \sqcup \iota_m)_*(\alpha) = \alpha \otimes \id{}$ are $[\psi( p_m \otimes \alpha(p_m) )]$ and $[\psi( \alpha(p_m) \otimes p_m )]$ respectively. But
	\[
		[\psi( p_m \otimes \alpha(p_m) )] = [p_m] \cdot [\alpha(p_m)] = [\alpha(p_m)] \in K_0(((D \otimes \bK)^{\otimes m})^\bT)
	\]
	by the definition of the ring structure on the $K$-theory group, which has $p_m$ as its unit. Since we obtain the same result for $[\psi( \alpha(p_m) \otimes p_m )]$, $G_D^\bT$ indeed has compatible inverses.
	
	It remains to be seen why $G_D^\bT$ is stable. To show this it suffices to prove that $\eqAut{\bT}{D \otimes \bK} \to \eqAut{\bT}{(D \otimes \bK)^{\otimes 2}}$ given by $\alpha \mapsto \alpha \otimes \id{}$ is a homotopy equivalence. But the map $H_l$ from \eqref{eqn:key_homotopy} is a homotopy between the two maps $\alpha \mapsto \beta(\tfrac{1}{2})^{-1} \circ (\alpha \otimes \id{}) \circ \beta(\tfrac{1}{2})$ and the identity. Since conjugation by $\beta(\tfrac{1}{2})$ is a homeomorphism, this shows that $\alpha \mapsto \alpha \otimes \id{}$ is a homotopy equivalence.
\end{proof}

\begin{corollary} \label{cor:coh_theory}
	The topological group $\eqAut{\bT}{D \otimes \bK}$ is an infinite loop space with associated cohomology theory $E_{D,\bT}^*(X)$ and 
	\[
		E_{D,\bT}^0(X) = [X, \eqAut{\bT}{D \otimes \bK}] \quad \text{and} \quad E_{D,\bT}^1(X) = [X, B\!\eqAut{\bT}{D \otimes \bK}]\ .
	\]
	The coefficients are given by 
	\[
		\check{E}^k_{D,\bT} \cong \begin{cases}
			0 & \text{if } k > 0 \ , \\
			GL_1(K_0^\bT(D)_+) & \text{if } k = 0\ , \\
			\pi_{-k-1}(U(D^\bT)) & \text{if } k < 0\ .
		\end{cases}
	\]
	In particular, isomorphism classes of $\bT$-equivariant locally trivial $C^*$-algebra bundles with fibres isomorphic to the $\bT$-algebra $D \otimes \bK$ over the finite CW-complex $X$ with trivial $\bT$-action form a group with respect to the fibrewise tensor product that is isomorphic to $E_{D,\bT}^1(X)$.
\end{corollary}

\begin{proof}
By Lem.~\ref{lem:stable_EHI-group} the diagram space $G_D^\bT$ is a commutative $\cI$-monoid. Stability implies that 
\(
	\eqAut{\bT}{D \otimes \bK} = G_D^\bT(\mathbf{1}) \to (G_D^\bT)_{h\cI} 
\)
is a homotopy equivalence \cite[Lem.~3.5]{paper:DadarlatP-UnitSpectra} . In particular, $\pi_0((G_D^\bT)_{h\cI})$ is a group (see also Lem.~\ref{lem:pi0_equiv}). Thus, $\eqAut{\bT}{D \otimes \bK}$ is an infinite loop space with respect to the operation induced by the tensor product. In particular, we have a delooping $B_\otimes\!\eqAut{\bT}{D \otimes \bK}$ as part of the $\Omega$-spectrum associated to $\eqAut{\bT}{D \otimes \bK}$. 

Let $B\!\eqAut{\bT}{D \otimes \bK}$ be the classifying space of $\eqAut{\bT}{D \otimes \bK}$ as a topological group (ie.\ using the composition instead of the tensor product). It was shown in \cite[Thm.~3.6]{paper:DadarlatP-UnitSpectra} that those two deloopings are homotopy equivalent as a result of the Eckmann-Hilton condition satisfied by $G_D^\bT$. This implies that the associated cohomology theory satisfies 
\[
		E_{D,\bT}^0(X) = [X, \eqAut{\bT}{D \otimes \bK}] \quad \text{and} \quad E_{D,\bT}^1(X) = [X, B\!\eqAut{\bT}{D \otimes \bK}]\ .
\]
The coefficients of $X \mapsto E_{D,\bT}^*(X)$ are isomorphic to the homotopy groups of $(G^\bT_D)_{h\cI} \simeq \eqAut{\bT}{D \otimes \bK}$. Hence,
\[
	\check{E}_{D,\bT}^k \cong \pi_{-k}(\eqAut{\bT}{D \otimes \bK})
\]
and $\check{E}_{D,\bT}^0$ is computed in Lem.~\ref{lem:pi0_equiv}, while for $k > 0$
\[
\check{E}_{D,\bT}^k \cong \pi_{-k}(\eqAut{\bT}{D \otimes \bK}) \cong \pi_{-k}(BU(D^\bT)) \cong \pi_{-k-1}(U(D^\bT))
\]
by Prop.~\ref{prop:Aut_Proj_equivalence}, Prop.~\ref{prop:Proj_BU} and the fact that $\Omega BU(D^\bT) \simeq U(D^\bT)$.
\end{proof}

\begin{remark}
	The proof of Cor.~\ref{cor:coh_theory} is easily adapted to $\eqAutId{\bT}{D \otimes \bK}$ showing that it is an infinite loop space as well. It gives rise to the cohomology theory $\bar{E}^*_{D,\bT}(X)$ associated to $G_{D,0}^\bT$ with 
	\[
		\bar{E}^0_{D,\bT}(X) = [X, \eqAutId{\bT}{D \otimes \bK}] \quad \text{and} \quad \bar{E}^1_{D,\bT}(X) = [X, B\!\eqAutId{\bT}{D \otimes \bK}]
	\]
	and coefficients
	\[
		\check{\bar{E}}^k_{D,\bT} \cong 
		\begin{cases}
			0 & \text{if } k \geq 0 \ ,\\
			\pi_{-k-1}(U(D^\bT)) & \text{if } k < 0\ .
		\end{cases}
	\]
\end{remark}

\begin{remark}
The reader familiar with \cite{paper:DadarlatP-UnitSpectra} will notice that Cor.~\ref{cor:coh_theory} did not use the fact that $G_D^\bT$ has compatible inverses. This property allowed the comparison with the unit spectrum of $K$-theory in \cite{paper:DadarlatP-UnitSpectra}. The question how this comparison generalises to $E^*_{D,\bT}(X)$ in the equivariant case will be central in future work.   
\end{remark}

Given a finite CW-complex $X$ some insight into $E^*_{D,\bT}(X)$ can be gained from the Atiyah-Hirzebruch spectral sequence with $E^2$-page:
\[
	E_2^{p,q} = H^p(X, \check{E}_{D,\bT}^q) \quad \Rightarrow \quad E^{p+q}_{D,\bT}(X)\ .
\]
A sketch of what the $E_2$-page looks like can be found in Fig.~\ref{fig:AHSS}. For degree reasons the $E_2$- and $E_3$-pages are the same and the first potentially non-trivial differential occurs on the $E_3$-page. 

\begin{figure}[htp]
\centering
\begin{tikzpicture}
  \matrix (m) [matrix of math nodes, nodes in empty cells,nodes={minimum width=1cm, minimum height=5ex,outer sep=-5pt}, column sep=1ex,row sep=1ex]{
                &   0  &  1  &  2 &  3 & \\
          0  &  H^0(X,G)  &  H^1(X,G)  & H^2(X,G) & H^3(X,G) \\
         -1\quad     &  0  & 0 &  0  & 0   \\
         -2\quad     &  H^0(X,\bddR)  & H^1(X,\bddR) &  H^2(X,\bddR)  & H^3(X,\bddR)\\
         -3\quad     &  0  & 0 &  0  & 0  \\
         -4\quad     &  H^0(X,\bddR^{0,\infty})  & H^1(X,\bddR^{0,\infty}) &  H^2(X,\bddR^{0,\infty})  & H^3(X,\bddR^{0,\infty}) \\
         -5\quad     &  0  & 0 &  0  & 0 \\ \\};
  \draw[-stealth,dashed] (m-4-2.south east) -- (m-6-5.north west) node[midway,above] {$d^3$};
  \draw[thick] (m-1-1.east) -- (m-8-1.east) ;
  \draw[thick] (m-1-1.south) -- (m-1-6.south) ;
\end{tikzpicture}
\caption{\label{fig:AHSS}Atiyah-Hirzebruch spectral sequence for $E_{D,\bT}^*(X)$. Here, $G=GL_1(K_0^\bT(D)_+)$. The coeffficients in degree $-2k$ for $k > 1$ depend on the chosen representation and are either $\bddR$, $\bddR^0$, $\bddR^\infty$ or $\bddR^0 \cap \bddR^\infty$ (see Lem.~\ref{thm:coefficients}).}	
\end{figure}

The homomorphism $\eqAut{\bT}{D \otimes \bK} \to GL_1(K_0^\bT(D))$ given by mapping $\alpha$ to its path-component can be used to associate to any $\bT$-equivariant locally trivial $D \otimes \bK$-bundle $\mathcal{A} \to X$ a bundle with fibre $K_0^\bT(D)$ and structure group $GL_1(K_0^\bT(D))$. The same result is obtained by applying the functor $K_0^\bT$ fibrewise. This gives the edge homomorphism of the spectral sequence on $E^1_{D,\bT}(X)$ and takes the form 
\[
	E^1_{D,\bT}(X) \to H^1(X, GL_1(K_0^\bT(D)))\ .
\] 


\begin{corollary} \label{cor:tori}
	Let $V$ be a $\bT$-representation with corresponding character polynomial $p_V = \sum_{i=0}^d a_i t^i$, let $D$ be the associated infinite UHF-algebra, let $n \in \N$ and let $\bT^n$ be the $n$-dimensional torus. Then we have 
	\begin{align*}
		E^1_{D,\bT}(\bT^n) & \cong H^1(\bT^n, G) \ \oplus\ H^3(\bT^n, \bddR) \ \oplus\ \bigoplus_{k=2}^\infty H^{2k+1}(\bT^n, \bddR^{0,\infty})\\ 
		& \cong \Z^n \oplus\! \bigoplus_{q \mid p_V \atop q \textrm{ prime}} \Z^n \ \oplus\ \extp^3 (\bddR)^n \ \oplus\ \bigoplus_{k=2}^\infty \extp^{2k+1} (\bddR^{0,\infty})^n	\ ,
	\end{align*}
	where the exterior algebras are taken over $\bddR$ and $\bddR^{0,\infty}$, respectively, and $\bddR^{0,\infty}$ is (with notation as in Thm.~\ref{thm:coefficients})
	\[
		\bddR^{0,\infty} = \begin{cases}
			\bddR & \text{if } a_0 > 1  \text{ and } a_d > 1\ , \\
			\bddR^0 & \text{if } a_0 = 1  \text{ and } a_d > 1\ , \\
			\bddR^{\infty} & \text{if } a_0 > 1  \text{ and } a_d = 1\ , \\
			\bddR^\infty \cap \bddR^0 & \text{if } a_0 = 1  \text{ and } a_d = 1\ .
		\end{cases}
	\]
\end{corollary}

\begin{proof}
	By \cite[Thm.~2.7]{paper:Arlettaz} the image of the differentials in the Atiyah-Hirze\-bruch spectral sequence are torsion subgroups. By Lem.~\ref{lem:units_in_loc} a unit in the ring $\Z[t,t^{-1},p_V(t)^{-1}]$ is of the form $\pm t^{k_0}q_1(t)^{k_1}\cdots q_r(t)^{k_r}$, where the polynomials $q_i$ are the distinct prime factors of $p_V$ and $k_0, k_1, \dots, k_r \in \Z$. This implies that 
	\[
		G = GL_1(K_0^\bT(D)_+) \cong \Z \oplus \!\bigoplus_{q \mid p_V \atop q \text{ prime}} \Z 
	\]
	with the first summand corresponding to $k_0$. In particular, $G$ is torsion free. Since the rings $\bddR$ and $\bddR^{0,\infty}$ are torsion free as well, the spectral sequence collapses on the $E_2$-page. The above observation also shows
	\[
		H^1(\bT^n,G) \cong H^1(\bT^n, \Z) \oplus \!\bigoplus_{q \mid p_V \atop q \text{ prime}} H^1(\bT^n,\Z) \cong \Z^n \oplus \bigoplus_{q \mid p_V \atop q \text{ prime}} \Z^n\ .
	\]
	Moreover, $H^k(\bT^n,R) \cong \extp^k \! R^n$ by the K\"unneth theorem. This implies the statement.
\end{proof}

\subsection{Comparison with the equivariant Brauer Group} \label{sec:BrauerGroup}
In this section we will take a closer look at the case where $V$ is the one-dimensional trivial representation, which implies $D = \C$ with the trivial $\bT$-action for the associated infinite UHF-algebra. We have $K_0^\bT(\C) \cong \Rep{\bT} = \Z[t,t^{-1}]$ and therefore 
\[
	GL_1(K_0^\bT(\C)_+) \cong \Z 
\]
generated by powers of $t$. Note that $\bddR \cong \Z$, while $\bddR^0 = \bddR^\infty = 0$. Hence, the coefficients of $E^*_{\C,\bT}(X)$ are given by
\[
	\check{E}^k_{\C,\bT} \cong \pi_{-k}(\eqAut{\bT}{\bK}) \cong \begin{cases}
		\Z & \text{if } k = 0\ , \\
		\Z & \text{if } k = -2\ , \\
		0 & \text{else} \ .
	\end{cases}
\]
Consider the inclusion map $\eqAutId{\bT}{\bK} \to \eqAut{\bT}{\bK}$ of the unit component and the embedding $\eqAut{\bT}{\bK} \to \Aut{\bK}$ of equivariant automorphism into all automorphisms. Let $\kappa \colon \eqAutId{\bT}{\bK} \to \Aut{\bK}$ be the composition. This map fits into the following commutative diagram:
\[
	\begin{tikzcd}
		\eqAutId{\bT}{\bK} \ar[r,"\kappa"] \ar[d,"\simeq" left] & \Aut{\bK} \ar[d,"\simeq"] \\
		\Proj{e}{\bK^\bT} \ar[r] & \Proj{e}{\bK}
	\end{tikzcd}
\]
Since $e \in \bK$ is a projection onto a one-dimensional trivial subrepresentation of $H = \ell^2(\Z) \otimes H_0$, the space $\Proj{e}{\bK^\bT} = \Proj{e}{C_0(\Z) \otimes \bK(H_0)}$ can be identified with $\Proj{e}{\bK(H_0)}$ and the bottom horizontal map is induced by the corner embedding $\bK(H_0) \to \bK(H) = \bK$. If we view $H_0$ as a closed subspace of $H$ and let $U(H_0) \to U(H)$ be the inclusion that is the identity on the complement of $H_0$, then the following diagram commutes:
\[
	\begin{tikzcd}[column sep=1.7cm]
		U(1) \times U(e^\perp H_0) \ar[d,"\id{U(1)} \times \left.\iota\right|_{U(e^\perp H_0)}"] \ar[r] & U(H_0) \ar[r,"u \mapsto ueu^*"] \ar[d,"\iota"] & \Proj{e}{\bK(H_0)} \ar[d] \\
		U(1) \times U(e^\perp H) \ar[r] & U(H) \ar[r,"u \mapsto ueu^*"] & \Proj{e}{\bK(H)}
	\end{tikzcd}
\]
The long exact sequence of homotopy groups then shows that the map $\Proj{e}{\bK^\bT} \to \Proj{e}{\bK}$ is a homotopy equivalence. Let $G_\C$ be the EH-$\cI$-group from \cite[Sec.~4.2]{paper:DadarlatP-UnitSpectra} for $D = \C$. The two group homomorphisms $\eqAutId{\bT}{\bK} \to \eqAut{\bT}{\bK}$ and $\eqAut{\bT}{\bK} \to \Aut{\bK}$ are compatible with forming tensor products and hence give rise to maps of commutative $\cI$-monoids
\[
	G_{\C,0}^\bT \to G_\C^\bT \qquad \text{and} \qquad G_\C^\bT \to G_\C\ .
\]
Inspecting the Atiyah-Hirzebruch spectral sequences for the cohomology theories $\bar{E}^*_{\C,\bT}(X)$ and $\bar{E}_\C^*(X)$ obtained from $G_{\C,0}^\bT$ and $G_{\C}$, respectively, we see that 
\[
	\bar{E}^k_{\C,\bT}(X) \cong H^{k+2}(X,\Z) \cong \bar{E}^k_{\C}(X)
\]
and our above observations show that $E^*_{\C,\bT}$ therefore naturally splits of the group $H^{k+2}(X,\Z)$. This implies that all the differentials in the Atiyah-Hirzebruch spectral sequence for $E^*_{\C,\bT}$ must vanish. With only two non-trivial rows on the $E_2$-page we obtain the following result: 
\begin{corollary} For $k \geq 0$ there is a natural isomorphism
	\[
		E^k_{\C,\bT}(X) \cong H^k(X,\Z) \oplus H^{k+2}(X,\Z)	\ .
	\]
\end{corollary}
The group $E^0_{\C,\bT}(X) \cong H^0(X,\Z) \oplus H^2(X,\Z)$ has a nice geometric interpretation not involving operator algebras: Given a map $f \colon X \to \eqAut{\bT}{\bK}$ define
 \[
 	L_f = \{ (x,\xi) \in X \times H\ |\ f(x)(e)\xi = \xi \}\ .
 \]
This is a line bundle over $X$ that carries a fibrewise $\bT$-action induced by the $\bT$-action on $H$. Homotopic maps induce isomorphic equivariant line bundles. Thus, $E^0_{\C,\bT}(X)$ classifies $\bT$-equivariant line bundles over $X$ up to isomorphism. The class in $H^2(X,\Z)$ is the Chern class $c_1(L_f)$, and the class in $H^0(X,\Z)$ is the $\Z$-valued function 
\(
	\chi \colon X \to \Z
\)
which determines the character of the $\bT$-action on each connected component. (Note that $\Endo{L_f} = L_f^* \otimes L_f = X \times \C$ and the $\bT$-action is given by $z \mapsto z^{\chi(x)}$.)

To understand $E^1_{\C,\bT}(X)$ we need to recall the definition of the equivariant Brauer group $\text{Br}_{\bT}(X)$ from \cite[Chap.~7]{book:RaeburnWilliams}: Its elements are equivalence classes $[A,\alpha]$ consisting of a continuous-trace $C^*$-algebra $A$ with spectrum~$X$ together with an action $\alpha \colon \bT \to \Aut{A}$ that commutes with the multiplication by $C(X)$, since we only consider the trivial $\bT$-action on $X$ in this section. The equivalence relation is $\bT$-equivariant $C(X)$-Morita equivalence (see \cite[Def.~7.2]{book:RaeburnWilliams} for details). We have a natural homomorphism
\begin{equation} \label{eqn:Brauer_iso}
	E^1_{\C,\bT}(X) \to \text{Br}_{\bT}(X)
\end{equation}
defined by mapping a $\bT$-equivariant locally trivial bundle $\mathcal{A} \to X$ with fibre~$\bK$ to the equivariant continuous-trace $C^*$-algebra $[C(X,\mathcal{A}), \alpha]$ obtained from the sections of the bundle together with the action induced by the fibrewise $\bT$-action on $\mathcal{A}$. By \cite[Thm.~7.20]{book:RaeburnWilliams} 
\[
	\text{Br}_{\bT}(X) \cong H^1(X, \widehat{\bT}) \oplus H^3(X,\Z) \oplus C(X,\underline{H}^2(\bT,\bT))\ ,
\]
where $\widehat{\bT} \cong \Z$ is the Pontrjagin dual of $\bT$. Moore has shown in \cite[Prop.~2.1]{paper:MooreExt} that for a compact connected Lie group $G$ there is a canonical isomorphism $H^2(G,\bT) \cong (\text{Tor}(\pi_1(G)))^\wedge$. In our case $\pi_1(\bT) \cong \Z$, which is torsion free. Hence, the third summand in the above decomposition vanishes and 
\[
	\text{Br}_{\bT}(X) \cong H^1(X, \Z) \oplus H^3(X,\Z)\ .
\]
The group homomorphism $\text{Br}_{\bT}(X) \to H^3(X,\Z)$ maps the equivalence class of a pair $[A,\alpha]$ to the Dixmier-Douady class of $A$. As discussed above, there is also a natural transformation $F \colon E^1_{\C,\bT}(X) \to \bar{E}^1_{\C}(X)$, which is surjective. These maps fit into the following commutative diagram:
\[
\begin{tikzcd}
	E^1_{\C,\bT}(X) \ar[r,"F"] \ar[d] & \bar{E}^1_{\C}(X) \ar[d,"\cong"] \\
		\text{Br}_{\bT}(X) \ar[r] & H^3(X,\Z) 
\end{tikzcd}
\]
The homomorphism $\text{Br}_{\bT}(X) \to H^1(X,\Z)$ is given by the Phillips-Raeburn obstruction. To understand its definition consider a continuous trace $C^*$-algebra $A$ with spectrum $\widehat{A} = X$ and equipped with a pointwise unitary $\bT$-action. By \cite[Cor.~1.2]{paper:Rosenberg} the action is then locally unitary and \cite[Prop.~2.2]{paper:PhillipsRaeburn-locunitary} implies that the restriction map
\[
	(A \rtimes \bT)^\wedge \to \widehat{A} 
\]
is a locally trivial $\widehat{\bT}$-bundle. Since $\widehat{\bT} \cong \Z$, the isomorphism class of this bundle is determined by a class in $H^1(X,\Z)$. This is the Phillips-Raeburn obstruction.

If $\mathcal{A} \to X$ is a locally trivial $\bK$-bundle with fibrewise $\bT$-action induced by $U_z$ as in \eqref{eqn:action_on_K} and $A = C(X,\mathcal{A})$, then $A \rtimes \bT$ is the section algebra of the locally trivial bundle with fibre $\bK \rtimes \bT$ associated to the principal $\eqAut{\bT}{\bK}$-bundle of $\mathcal{A}$ using the homomorphism 
\[
	\eqAut{\bT}{\bK} \to \Aut{\bK \rtimes \bT} \quad , \quad \alpha \mapsto \alpha \rtimes \id{}
\]
Since the action of $\bT$ on $\bK$ is unitarily implemented, the map $a \mapsto a\,U$ defines an isomorphism $\bK \rtimes \bT \to \bK \otimes C^*(\bT) \cong \bK \otimes C_0(\Z)$. This isomorphism intertwines $\alpha \rtimes \id{}$ with $\alpha \otimes s_{\theta(\alpha)} \in \Aut{\bK \otimes C_0(\Z)}$, where $s_n \in \Aut{C_0(\Z)}$ is given by $s_n(f)(x) = f(x - n)$ and $\theta$ is the homomorphism from \eqref{eqn:pi0-hom_for_K}. Thus, the cocycle defining the principal $\Z$-bundle $(A \rtimes \bT)^\wedge \to \widehat{A}$ is associated to the principal $\eqAut{\bT}{\bK}$-bundle via $\theta$. But since $\theta$ induces the homomorphism $E^1_{\C,\bT}(X) \to H^1(X,\Z)$ the following diagram commutes  
\[
\begin{tikzcd}
	E^1_{\C,\bT}(X) \ar[dr,"\theta"] \ar[d] \\
		\text{Br}_{\bT}(X) \ar[r] & H^1(X,\Z) 
\end{tikzcd}
\]
We summarise our observations in the following theorem: 
\begin{theorem} \label{thm:BrauerGroup}
	The group homomorphism $E^1_{\C,\bT}(X) \to \text{\rm Br}_{\bT}(X)$ in \eqref{eqn:Brauer_iso} is an isomorphism such that the following diagram of natural isomorphisms commutes: 
	\[
	\begin{tikzcd}
		E^1_{\C,\bT}(X)\ar[r] \ar[dr] & \text{\rm Br}_{\bT}(X) \ar[d] \\
		& H^1(X,\Z) \oplus H^3(X,\Z)		
	\end{tikzcd}
	\]
\end{theorem}

In particular, any class in the equivariant Brauer group $\text{Br}_{\bT}(X)$ can be realised as the section algebra of a $\bT$-equivariant locally trivial bundle over $X$ with fibre $\bK$. Its Phillips-Raeburn obstruction corresponds to the isomorphism class of the principal $\Z$-bundle defined by applying $\pi_0$ to each fibre of the principal $\eqAut{\bT}{\bK}$-bundle and the Dixmier-Douady class is the one of the $\bK$-bundle after forgetting the $\bT$-action.

Note that, since the $\bT$-action on $X$ is trivial, we also have the following isomorphism using the K\"unneth theorem
\begin{align*}
	H^3_{\bT}(X,\Z) &\cong H^3(B\bT \times X,\Z) \\ 
	&\cong H^0(B\bT,\Z) \otimes H^3(X,\Z) \oplus H^2(B\bT,\Z) \otimes H^1(X,\Z) \\
	&\cong H^1(X,\Z) \oplus H^3(X,\Z)
\end{align*}
Hence, $E^1_{\C,\bT}(X) \cong H^3_{\bT}(X,\Z)$, which is another indication of a potential interpretation of $E^*_{D,\bT}(X)$ in terms of equivariant stable homotopy theory.

\section*{Acknowledgements}
	The authors would like to thank the anonymous referee for the suggestion to investigate the link between tensorial $\JiangSu$-absorption of the fixed-point algebra $D^\bT$ and the case distinction in Thm.~\ref{thm:coefficients} (see Rem.~\ref{rem:Z-absorbtion}). 

\appendix

\section{Polynomials with \\ non-negative integer coefficients} \label{sec:appendix}
The key observation about polynomials with non-negative integer coefficients and non-zero constant term that we will use frequently in this section is the following: If $p(t) = \sum_{i=0}^d a_i t^i$ is such a polynomial and $0 = k_0 < k_1 < k_2 < \dots < k_r = d$ are those powers of $t$ such that $a_{k_j} > 0$, then the non-zero coefficients of $p(t)^m = \sum_{i=0}^{md} b_i t^i$ are $b_{\ell}$ with  
\[
	\ell = \sum_{i=1}^r x_i k_i
\]
for all $r$-tuples $x = (x_1, \dots, x_r)$ of non-negative integers with $\sum_{i=1}^r x_i \leq m$.

\begin{definition} \label{def:large_middle}
	Let $d \geq 1$. A polynomial $p(t) = \sum_{i = 0}^d a_i t^i$ of degree $d$ with non-negative integer coefficients is said to have \emph{large middle coefficients} if $a_j \neq 1$ for all $j \in \{1, \dots, d-1\}$, $a_d \neq 0$ and $a_0 \neq 0$.
\end{definition}

\begin{lemma} \label{lem:large_middle}
	Let $d \geq 1$ and let $p \in \Z[t]$ be any polynomial of degree $d$ with non-negative coefficients and non-zero constant term, then there is an $N \in \N$ such that $p^N$ has large middle coefficients.
\end{lemma}

\begin{proof}
	Let $p(t) = \sum_{i=0}^d a_i t^i$ and let $0 = k_0 < k_1 < k_2 < \dots < k_r = d$ be those powers of $t$ such that $a_{k_j} > 0$. We claim that the statement holds for $N = d$. This is true for $d = 1$, since all polynomials $a_0 + a_1t$ with $a_0 \neq 0$ and $a_1 \neq 0$ have large middle coefficents. 
		
	Let $p(t)^{d-1} = \sum_{i=0}^{(d-1)d} b_i t^i$ and let $p(t)^d = \sum_{i=0}^{d^2} c_i t^i$. If we have a non-zero coefficient $c_{\ell}$ corresponding to an $r$-tuple $(y_1, \dots, y_r)$ with at least two non-zero entries $y_s$ and $y_t$, then $c_{\ell}$ is a sum of at least two terms in the product $p(t)^{d-1} \cdot p(t)$. For example, let $\ell_1$ correspond to $(y_1, \dots, y_s-1, \dots, y_r)$ and $\ell_2$ to $(y_1, \dots, y_t - 1, \dots, y_r)$. Then $b_{\ell_1}a_{k_s}t^{\ell}$ and $b_{\ell_2}a_{k_t}t^{\ell}$ are both summands in $p(t)^d$. Therefore all $c_{\ell}$ corresponding to $r$-tuples with at least two non-zero entries satisfy $c_{\ell} \geq 2$. 
	
	The remaining terms to be considered are the coefficients $c_{\ell}$ corresponding to $r$-tuples $y = (0,\dots, 0, y_s, 0, \dots, 0)$ with one non-zero entry. If $y_s < d$, then there are at least two contributions to $c_{\ell}$: $a_0b_{y_sk_s} t^{y_sk_s}$ and $a_{k_s}b_{(y_s-1)k_s}t^{y_sk_s}$. If $y_s = d$ with $s < r$, then 
	\[
		d k_s = k_{s+1} k_s + (d-k_{s+1})k_s\ .
	\] 
	Note that $d- k_{s+1} \geq 0$ and $d - k_{s+1} + k_s = d - (k_{s+1} - k_s) \leq d$. Therefore $y' = (0, \dots, 0, d-k_{s+1}, k_s, 0, \dots,0)$ with entries in position $s$ and $s+1$ is a valid $r$-tuple, which gives rise to the same coefficient $c_\ell$ as $y$. Thus, we have $c_{\ell} \geq 2$ in this case as well.
\end{proof}

The reason for the terminology is the following lemma, which also encapsulates our main application for polynomials with large middle coefficients. 
\begin{lemma} \label{lem:coefficients_grow}
	Let $p \in \Z[t]$ be a polynomial of degree $d \geq 1$ with large middle coefficients. If $p(t)^m = \sum_{i=0}^{md} b_i t^i$, then either $b_i > m$ or $b_i = 0$ for all $i \in \{1, \dots, md-1\}$. Moreover, if $a_0 > 1$, then  $b_0 > m$. Likewise, if $a_{d} > 1$, then $b_{md} > m$. 
\end{lemma}

\begin{proof}
	We will prove the statement by induction over $m$. It is true for $m = 1$. Let $p(t) = \sum_{i=0}^d a_i t^i$ and let $0 = k_0 < k_1 < k_2 < \dots < k_r = d$ be the powers of $t$ such that $a_{k_j} > 0$. By our key observation the only non-zero coefficients $b_{\ell}$ occur for $\ell = \sum_{i=1}^r x_i k_i$ with $\sum_{i=1}^r x_i \leq m$ and by induction we have $b_{\ell} > m$ if $\ell \in \{1,\dots, md-1\}$. Now consider
	\[
		p(t)^{m+1} = \left( \sum_{i=0}^d a_i t^i\right) \cdot  \left( \sum_{j=0}^{md} b_j t^j \right) = \sum_{i=0}^d a_i \left( \sum_{j=0}^{md} b_j t^{i+j} \right) \ .
	\]
	For $i \in \{1,\dots, d-1\}$ and $j \in \{1, \dots, md-1\}$ we either get $a_i b_j = 0$ or $a_i b_j > 2m \geq (m+1)$, so the only potentially problematic summands occur if $i = 0$, $i =d$, $j = 0$ or $j = md$.

	Let us first consider the case $i= 0$. If $a_0 > 1$, then there is no problem, which shows the second part of the statement for this case. Otherwise $a_0 = 1$ and if $a_0 b_{\ell} = b_{\ell} \neq 0$ corresponds to $(x_1, \dots, x_r)$ with at least one $x_s > 0$, then there is another summand $a_{k_s}b_{\ell'}t^{k_s+\ell'}$ producing the same power of $t$ with $\ell'$ corresponding to $(x_1, \dots, x_s-1, \dots, x_r)$. Hence, the coefficient $c_\ell$ of $t^\ell$ in $p(t)^{m+1}$ satisfies $c_\ell \geq b_{\ell} + a_{k_s}b_{\ell'} > m + 1$. The case $a_d b_{\ell}$ with $0 < \ell < md$ is shown similarly. 

	Now consider the summands of $p(t)^{m+1}$ of the form $a_{k_i} b_0 t^{k_i}$. By the above argument for $i \in \{0,d\}$ we only need to consider $i \in \{1, \dots, r-1\}$. There is at least one other summand that produces the same power of $t$, for example $a_0 b_{\ell} t^\ell$ with $\ell$ corresponding to $x = (0,\dots, 1, \dots,0)$, where $1$ sits in the $i$th position. Hence, the coefficient $c_{k_i}$ of $t^{k_i}$ in $p(t)^{m+1}$ satisfies $c_{k_i} \geq a_{k_i} b_0 + a_0 b_{\ell} > 2 + m > m+1$.
	
	Finally consider the summands $a_{k_i} b_{md} t^{k_i + md}$ for $i \in \{1, \dots, r-1\}$ and note that $\ell = md$ corresponds to $x = (0,\dots, 0,m)$. Again there is at least one other summand producing the same power of $t$, for example the term $a_d b_{k_i + (m-1)d} t^{k_i + md}$ where $k_i + (m-1)d$ corresponds to the $r$-tuple $x = (0, \dots, 0, 1,0, \dots, m -1)$, where the $1$ sits again in the $i$th position. Hence, the coefficient $c_{k_i + md}$ of $t^{k_i + md}$ in $p(t)^{m+1}$ satisfies $c_{k_i + md} \geq a_{k_i} b_{md} + a_d b_{k_i + (m-1)d} > 2 + m > m+1$.
\end{proof}

\bibliographystyle{habbrv}
{\footnotesize
\bibliography{EquivariantAutomorphisms}
}
\end{document}